\documentclass[reqno]{amsart}
\usepackage[utf8]{inputenc}
\usepackage[paperwidth=8.5in, paperheight=11in]{geometry}

\usepackage{amsmath}
\usepackage{hyperref}
 \usepackage{amscd}
 \usepackage{url}
 \usepackage{amsthm}
 \usepackage{amsfonts}
 \usepackage{amssymb}
 \usepackage{mathrsfs}  
 \usepackage{enumerate}
 \usepackage{color}
 \usepackage{tikz-cd}
 \usepackage{mathtools}

\theoremstyle{plain}
 \newtheorem{thm}{Theorem}
 \numberwithin{thm}{subsection} % important bit

 \numberwithin{equation}{subsection}
 \newtheorem{prop}[thm]{Proposition}
 \newtheorem{lem}[thm]{Lemma}

 \theoremstyle{definition}
 \newtheorem{defi}[thm]{Definition}

 \theoremstyle{remark}

 \newtheorem{rmk}[thm]{Remark}

 \newtheoremstyle{warnings}{}{}{\color{red}}{}{\color{red}\bfseries}{.}{ }{}
 \theoremstyle{warnings}

\newcommand{\colim}{\operatorname{colim}}

\newcommand{\Z}{\mathbb{Z}}

\newcommand{\C}{\mathscr{C}}

\newcommand{\R}{\mathbb{R}}

\newcommand{\Q}{\mathbb{Q}}

\newcommand{\X}{\mathscr{X}}
\newcommand{\Xf}{\mathfrak{X}}
\newcommand{\Yf}{\mathfrak{Y}}

\newcommand{\Fs}{\mathscr{F}}

\newcommand{\Y}{\mathcal{Y}}

\renewcommand{\P}{\mathbb{P}}

\newcommand{\A}{\mathbb{A}}

\newcommand{\m}{\mathfrak{m}}

\newcommand{\ov}{\overline}

\newcommand{\Spec}{\operatorname{Spec}}

\newcommand{\Hom}{\operatorname{Hom}}

\newcommand{\Mod}{\operatorname{Mod}}
\newcommand{\ModDesc}{\operatorname{ModDesc}}
\newcommand{\Alg}{\operatorname{Alg}}

\newcommand{\Eq}{\operatorname{Eq}}

\newcommand{\Res}{\operatorname{Res}}

\newcommand{\subsetq}{\subseteq}

\title{A Topology on Points on Stacks}
\author{Atticus Christensen}
\address{Department of Mathematics, Massachusetts Institute of Technology, Cambridge, MA 02139-4307, USA}
\email{atticusc@mit.edu}
\urladdr{http://math.mit.edu/~atticusc/}
\thanks{The work in this thesis was partly supported by Simons Foundation grants
  \#402472 (to Bjorn Poonen) and \#550033, and by National Science Foundation grant DMS-1601946.}
\date{\today}

\hypersetup{ colorlinks=true, linkcolor=black, filecolor=magenta,
  urlcolor=cyan, }

\begin{document}
\begin{abstract}
  For a variety over certain topological rings $R$, like $\Z_p$ or $\mathbb{C}$, there is a well-studied way to topologize the $R$-points on the variety. In this paper, we generalize this definition to algebraic stacks. For an algebraic stack $\Xf$ over many topological rings $R$, we define a topology on the isomorphism classes of $R$-points of $\Xf$. We prove expected properties of the resulting topological spaces including functoriality. Then, we extend the definition to the case when $R$ is the ring of adeles of some global field. Finally, we use this last definition to strengthen the local-global compatibility for stacky curves of Bhargava--Poonen to a strong approximation result.
 \end{abstract}
\maketitle
\tableofcontents
\section{Introduction}

To study a variety $X$ defined over $\Q$, it is often useful to study the base change of this variety to the various completions of $\Q$.  One benefit of this is that for any place $p$ of $\Q$, the points $X(\Q_p)$ can be endowed with the structure of a topological space.

In this paper, we will generalize this definition to the case of stacks. When $\Xf$ is an algebraic stack over $\Q_p$ we define a natural topology on the isomorphisms classes of $\Q_p$ points, which we denote by $\Xf(\Q_p)$. We in fact give a define a topology more generally: for a large class of topological rings $R$ and for finite type algebraic stacks $\Xf$ over $R$, we define a natural topology on $\Xf(R)$.

The idea for the definition comes from the fact that for $X$ and $Y$ two varieties over $\Q_p$ and $X\to Y$ a smooth morphism, the map $X(\Q_p)\to Y(\Q_p)$ is open, and in particular a quotient onto its image. If one assumes that a smooth map of stacks should have the same property one arrives at the definition given in Section \ref{sec:def}.

In order for the definition to make sense, we establish in Section \ref{sec:cocycle} that for any algebraic stack $\Xf$ over $R$ and $x\in \Xf(R)$, there is a smooth cover $Z\to \Xf$ such $x$ lifts to a point of $Z(R)$. To do this we introduce ``cocycle'' spaces. If we are given a smooth cover $X\to \Xf$ by a scheme and we wish to define a map $T\to \Xf$ for a scheme $T$, one strategy is to find a flat cover $T'\to T$ and to give a map $T'\to Z$ and certain descent data; these cocycle spaces parameterize such maps $T'\to Z$ together with descent data.

In Section \ref{sec:func}, we prove expected functorialities of these spaces.

In Section \ref{sec:algspace} and \ref{sec:algstack}, we prove basic properties of $\Xf(R)$ for particularly well-behaved topological rings. In particular we establish that smooth maps between stacks induce open maps on $R$-points, that for separated algebraic stacks $\Xf$ over $\Q_p$, $\Xf(\Q_p)$ is Hausdorff, and that for proper Deligne-Mumford stacks $\Xf$ with finite diagonal over $\Q_p$, $\Xf(\Q_p)$ is furthermore compact.

Finally, in Section \ref{sec:adeles} we explain how if $\A$ is a ring of adeles of a global field and $\Xf$ a finitely presented stack over $\A$ we can define a topology on $\Xf(\A)$. We conclude this section by giving an application of the theory, strengthening a result of Bhargava and Poonen (definitions can be found in Section \ref{sec:adeles})

\begin{thm}[Strong approximation for stacky curves]
  Let $k$ be a number field. Let $\Xf$ be a stacky curve over $k$ of genus llss than $1/2$. Let $\A'$ be obtained by removing one place from the adeles of $K$. Then \[X(k)\to X(\A')\] has dense image.
\end{thm}

A related topological space associated to a algebraic stack over topological field has been defined by Ulirsch in \cite{Ulirsch}. The major difference between the topological space defined there and the one in this paper is that the space defined in \cite{Ulirsch} identifies two points if they become isomorphic after a non-archimedean field extension. No such identification happens in this paper.
\section{Background}
The theory of algebraic spaces and algebraic stacks begins with the theory of descent. We begin by giving a brief review of this theory.

Let $R$ be a ring. Let $S$ be an $R$-algebra. For an $R$-module $M$, we may extend scalars $M\otimes_R S$ to obtain an $S$-module. Set $S_2=S\otimes_R S$, and note there are two $R$-algebra maps $i_1,i_2:S\to S\otimes_R S$ given by $i_1(s)=s\otimes 1$ and $i_2(s)=1\otimes s$. There is a canonical isomorphism $(M\otimes_R S)\otimes_{S, i_1}(S\otimes_R S)\to (M\otimes_R S)\otimes_{i_2,S}(S\otimes_R S)$ coming from the fact that both $(M\otimes_R S)\otimes_{S, i_1}(S\otimes_R S)$ and $(M\otimes_R S)\otimes_{S, i_2}(S\otimes_R S)$ are canonically isomorphic to $M\otimes_R (S\otimes_R S)$.

Thus an $S$-module $N$ can only potentially be obtained from an extension of scalars from an $R$-module if there is an isomorphism $N\otimes_{S, i_1}(S\otimes_R S)\to N\otimes_{S, i_2}(S\otimes_R S)$.

With this in mind, we define:
\begin{defi}
  Let $R$ be a ring and $S$ an $R$-algebra. An $S$-module with descent datum to $R$ is a pair $(N,\psi)$ of an $S$-module $N$ and an isomorphism $\psi:N\otimes_{S, i_1}(S\otimes_R S)\to N\otimes_{S, i_2}(S\otimes_R S)$.
\end{defi}

Using these we define a category of modules with descent data.
\begin{defi}
  Let $R$ be a ring and $S$ an $R$-algebra. Let $\mathrm{ModDesc}_{R\to S}$ denote the category of $S$-modules with descent datum to $R$, where the objects are $S$-modules with descent datum to $R$, and a morphism $(N_1,\psi_1)\to (N_2,\psi_2)$ is $S$-module morphism $\phi:N_1\to N_2$ such the following diagram is commutative:

  \[
    \begin{tikzcd}
      N_1\otimes_{S, i_1}(S\otimes_R S)\ar[r,"\psi_1"]\ar[d,"\phi"]&N_1\otimes_{S, i_2}(S\otimes_R S)\ar[d,"\phi"]\\
      N_2\otimes_{S, i_1}(S\otimes_R S)\ar[r,"\psi_1"]&N_2\otimes_{S, i_2}(S\otimes_R S)
    \end{tikzcd}.
  \]
\end{defi}

If $\Mod_R$ is the category of $R$-modules, we have already seen there is a natural functor $\Mod_R\to \ModDesc_{R\to S}$.

\begin{thm}[Faithfully flat descent,{\cite[\href{https://stacks.math.columbia.edu/tag/032M}{Tag 032M}]{stacks-project}}  ]\label{thm:faithfulflat}
  If $R\to S$ is a faithfully flat map of rings, $\Mod_R\to \ModDesc_{R\to s}$ is an equivalence of categories.

  Furthermore, we can describe a quasi-inverse. Begin with $(N,\psi)\in \ModDesc_{R\to S}$. We have two maps $N\to N\otimes_{S, i_2}(S\otimes_R S)$: the first is the natural map onto the first factor and the second is the composite $N\to N\otimes_{S, i_1}(S\otimes_R S)\xrightarrow{\psi} N\otimes_{S, i_2}(S\otimes_R S)$. We associate the $R$-module $\Eq(N\rightrightarrows N\otimes_{S, i_2}(S\otimes_R S)$
\end{thm}

\subsection{Sheaves}
We now turn our attention to sheaves. First let $\C=\Alg_R^{op}$ be the opposite category to the category of $R$-algebras. For an $R$-algebra $S$, let $\Spec S$ denote the associated object in $\C$.

 We make $\C$ a site by saying that $\Spec S_2 \to\Spec S_1$ is a covering if the associated ring map $S_1\to S_2$ is faithfully flat. To an $R$ module $M$, we can associate a presheaf on this site denote by $\widetilde{M}$ where $(\Spec S)(\widetilde{M})=M\otimes_R S$. Then Theorem \ref{thm:faithfulflat} implies that $\widetilde{M}$ is a sheaf.

Now let $S$ be a scheme. Let $S_{\mathrm{fppf}}$ be the fppf site over $S$: the objects are $S$-schemes, morphisms are scheme morphisms over $S$, and coverings are faithfully flat and locally finitely presented morphisms. For any $S$-scheme $X$,  $F_X(Y)=\Hom_S(Y,X)$ defines a presheaf $F_X$ on $S_{\mathrm{fppf}}$. Theorem \ref{thm:faithfulflat} implies:
\begin{thm}[{\cite[\href{https://stacks.math.columbia.edu/tag/023Q}{Tag 023Q}]{stacks-project}}]
  The presheaf $F_X$ is a sheaf on $S_{\mathrm{fppf}}$.
\end{thm}

Let us expand upon what it means for $F_X$ to be a sheaf. Let $T$ be another $X$-scheme, and let $T'\to T$ be a faithfully flat cover. The question we seek to answer is: given a map $T'\to X$, when is there a map $T\to X$ making the following commute:
\[
  \begin{tikzcd}
    T'\ar[d]\ar[dr]& \\
    T\ar[r]&X.
  \end{tikzcd}
  \]

That $F_X$ is a sheaf implies that $F_X(T)$ is the equalizer in the sequence $F_X(T)\to F_X(T')\rightrightarrows F_X(T'\times_T T')$. Then this gives that a map $T'\to X$ comes from a map $T\to X$ if and only if two compositions $T'\times_T T'\xrightarrow{\pi_i}T'\to X$ are equal.

We give another consequence of the fact that $F_X$ is a sheaf. Let $X'\to X$ be a faithfully flat cover. Let $T$ be an $S$-scheme with a map $T\to X$. Define $T'=T\times_X X'$; then we get an associated map $T'\to X'$. The two composites $T'\times_T T'\xrightarrow{\pi_i}T'\to X'$ define a map $T'\times_T T'\to X'\times_X X'$. This map makes the following diagram commute (where the vertical arrows are projections):

\begin{equation}\label{eqn:ff}
  \begin{tikzcd}
    T'\times_T T'\ar[xshift=3pt]{d} \ar[xshift=-3pt]{d}\arrow{r}&  X'\times_{X}X' \ar[xshift=3pt]{d} \ar[xshift=-3pt]{d}\\
    T'\arrow{r}& X'.
  \end{tikzcd}
\end{equation}

Again we can go backwards: another way we can specify a map $T\to X$ is to choose an fppf cover $T'\to X'$ and a map $T'\times_T T'\to X'\times_X X'$ making Equation \ref{eqn:ff} commute.

That this defines a map $T\to X$ again follows from the fact that $F_X$ is a sheaf. This strategy will inform our method of defining maps to stacks.

\subsection{Algebraic spaces and stacks}\label{sec:stackssetup}

Let $X$ be an algebraic space over $S$. As algebraic spaces are quotients of schemes by \'etale equivalence relations, there is a scheme $Y$ with surjective \'etale map to $X$. If we let $R=Y\times_X Y$, then $R$ is a subscheme of $Y\times Y$ whose two projections $R\rightrightarrows Y$ are \'etale. Specifying the data $R\rightrightarrows Y$ specifies $X$ as the quotient of the sheaf $F_Y$ by the equivalence relation defined by $R$.

If $T$ is an $S$-scheme, one way of giving a map $T\to X$ is to find an \'etale cover $T'\to T$ with map $T'\to Y$ and a map $T'\times_T T'\to Y\times_X Y$ making the following diagram commute:
\[
  \begin{tikzcd}
    T'\times_T T'\ar[xshift=3pt]{d} \ar[xshift=-3pt]{d}\arrow{r}&  R \ar[xshift=3pt]{d} \ar[xshift=-3pt]{d}\\
    T'\arrow{r}& Y.
  \end{tikzcd}
\] Thus we can define a map $T\to X$ using only maps between schemes.

Now we turn to algebraic stacks. Let $\Xf$ be an algebraic stack over $S$. Then there is a scheme $X$ with a smooth cover $\pi: X\to \Xf$. Then $R=X\times_{\Xf}X$ is an algebraic space with two projections $R\rightrightarrows X$.

Let $T$ is an $S$-scheme. We may begin to give a map $f:T\to \Xf$ similarly as before, by giving a smooth cover $T'\to T$, a map $T'\to X$, and a map $f:T'\times_T T'\to X\times_{\Xf}X$ making the following again commute:
\[
  \begin{tikzcd}
    T'\times_T T'\ar[xshift=3pt]{d} \ar[xshift=-3pt]{d}\arrow{r}& R \ar[xshift=3pt]{d} \ar[xshift=-3pt]{d}\\
    T'\arrow{r}& X.
  \end{tikzcd}
\] 

However, for stacks this is not sufficient.

The map $T'\times_T T'\to X\times_{\Xf}X$ defines a map $(T'\times_T T')\times_S (T'\times_T T')\to (X\times_{\Xf}X)\times_S (X\times_{\Xf}X)$, which we may restrict to a map $(T'\times_T T')\times_{T'} (T'\times_T T')\to (X\times_{\Xf}X)\times_S (X\times_{\Xf}X)$, and the commutativity of the last diagram guarantees the image lands in $(X\times_{\Xf}X)\times_X (X\times_{\Xf}X)$. Thus this defines a map $(T'\times_T T')\times_{T'} (T'\times_T T')\to (X\times_{\Xf}X)\times_X (X\times_{\Xf}X)$, but this is canonically isomorphic to a morphism $T'\times_T T'\times_T T'\to X\times_{\Xf}X\times_{\Xf}X$.

The condition for the maps $T'\to X$ and $T'\times_T T'\to R$ to define a map $T\to \Xf$ is for the following diagram to commute (where the vertical arrows are the projections): 
\[\begin{tikzcd}
    T'\times_T T'\times_TT'\ar[xshift=6pt]{d} \ar[xshift=-6pt]{d}\ar{d}\arrow{r}&X\times_{\Xf}X\times_{\Xf}X\ar[xshift=6pt]{d} \ar[xshift=-6pt]{d}\ar{d}\\
    T'\times_TT'\arrow{r}&  X\times_{\Xf}X 
  \end{tikzcd}.\]

\begin{rmk}
  For stacks the map $T'\times_T T'\to R$ is additional information, whereas for algebraic spaces the question is only whether or not one exists.
\end{rmk}

\section{Conventions and notations}
Throughout the text all topological rings are assumed to be Hausdorff. When we work over a ring $R$, all algebraic spaces, schemes, and stacks are finite-type over $R$.

For $S'\to S$ a map of schemes and $\Fs$ an fppf sheaf on the category of $S'$-schemes, denote by $\Res_{S'/S}\Fs$ the sheaf such that for a $S$-scheme $T$ we have $(\Res_{S'/S}\Fs)(T)\coloneqq \Fs(T\times_S S')$. If $\Fs$ is an fppf sheaf on the category of $S$-schemes and $\Fs_{S'}$ its pullback to the category for $S'$-schemes, by abuse of notation we write $\Res_{S'/S}\Fs$ for $\Res_{S'/S}\Fs_{S'}$. For a $S$-scheme, $X$, we will identify $X$ with its associated fppf sheaf, allowing us to write $\Res_{S'/S}X$. Note that the sheaves obtained in this way are not necessarily representable.

\section{Topologies on points of varieties}\label{sec:pvar}
In this section $R$ is a topological ring. For a large class of $R$, we will explain how for any finite-type $R$-scheme $X$ we may define a topology on $X(R)$. The topology will satisfy the following axioms:

\begin{enumerate}[(1)]
\item If $X=\A^1_R$, then $X(R)=R$.
\item If $X\to Y$ is a morphism of finite-type $R$-schemes, then $X(R)\to Y(R)$ is continuous.
\item If $X$ is a finite-type $R$-scheme and $Y\hookrightarrow X$ is a closed immersion, then $Y(R)\subseteq X(R)$ is a closed subset with the subspace topology.
\item If $X$ and $Y$ are two finite-type $R$-schemes, then $(X\times_R Y)(R)=X(R)\times Y(R)$.
\item If $X$ is a finite-type $R$-scheme, and $Y\hookrightarrow X$ is an open immersion, then $Y(R)\subseteq X(R)$ is an open subset with the subspace topology.
\end{enumerate}

\begin{defi}\label{def:good}
  Let $R$ be a topological ring. A topology $X(R)$ for every finite-type $R$-scheme $X$ is called a {\em excellent topologization of $R$-points of finite-type $R$-schemes} if $(1)-(5)$ hold.

  Let $R$ be a topological ring. A topology $X(R)$ for every finite-type $R$-scheme $X$ is called a {\em good topologization of $R$-points of finite-type $R$-schemes} if $(1)-(4)$ hold.
\end{defi}

\begin{defi}\label{def:contint}
  If $R$ is topological ring and we give $R^\times$ the subset topology, then we say $R$ is {\em continuously invertible} if the map $R^\times\to R^\times$ given by $r\mapsto 1/r$ is continuous.
\end{defi}
This condition implies that if $f\in R[x_1,\ldots, x_n]$ and $U\subseteq R^n$ is such that $f(U)\subseteq R^\times$, then the map $U\to R$ given by $1/f(x_1,\ldots,x_n)$ is continuous.

By local topological ring, we mean a topological ring that is local as a ring. We will first explain how if $R$ is a continuously invertible topological ring, then $R$ has a unique excellent topologization of $R$-points of finite-type $R$-schemes.

First if $X=\A^n$ is affine space, then the space $X(R)=R^n$ we give product topology, which is required by properties $(1)$ and $(4)$. If $X\subseteq \A^n$ is a closed subscheme of affine space, we define the topology as the subspace topology $X(R)\subseteq R^n$, which is required by property $(3)$. We must show this is independent of the affine embedding.

Suppose $X$ has two embeddings into affine spaces, $i_1:X\hookrightarrow \A^n$ and $i_2:X\hookrightarrow \A^m$. Then there is a morphism $r_1:\A^n \to \A^m$ such that $r_1\circ i_1 =i_2$. Similarly, there is a morphism $r_2:\A^m\to \A^n$ such that $r_2\circ i_2=i_1$. As polynomial maps are continuous, $r_1$ and $r_2$ induce continuous maps between the two topologies on $X(R)$ defined by the two embeddings. Thus the topologies must be the same.

Next if $X=\Spec A$ for an finitely generated $R$-algebra $A$, and $f\in A$ and $U=\Spec A[1/f]\subseteq X$ is a distinguished open, we would like to see that the topology on $U(R)$ induced as a subspace of $X(R)$ is same as the topology on $U(R)$ when $U$ is viewed as an affine scheme. Choose a closed embedding $X\hookrightarrow \A^n$ and let $x_1,\ldots,x_n$ be the coordinates on $\A^n$. Let $U_{\mathrm{sub}}$ be the set $U(R)$ equipped with the subspace topology as a subspace $U(R)\subseteq X(R)\subseteq R^n$. Let $U_{\mathrm{aff}}$  be the set $U(R)$ equipped with the topology viewing $U$ as an affine scheme. Note that $U\subseteq \A^{n+1}$ as a closed subscheme since $U\cong \Spec A[x_{n+1}]/(x_{n+1}f-1)$. The map $(x_1,\ldots,x_n)\mapsto (x_1,\ldots,x_n,1/(f(x_1,\ldots,x_n))$ is defined on an open subset of $\A^n$ mapping to $\A^{n+1}$ and restricts to a continuous map $U_{\mathrm{sub}}\to U_{\mathrm{aff}}$. The projection onto the first $n$ coordinates gives a map $\A^{n+1}\to \A^n$ which gives a continuous map $U_{\mathrm{aff}}\to U_{\mathrm{sub}}$. Thus both must have the same topologies.

Let $X$ be such a finite-type scheme. First note any $R$-point of $X$ is contained in an affine subset of $X$. On each affine $U\subseteq X$, give $U(R)$ topology as an affine subset; this is required by $(5)$s. We must check that for two affines, $U,V$, the topologies agree on $U\cap V$. But $U\cap V$ is covered by opens which are distinguished in both $U$ and $V$. On these distinguished opens, we have shown that the topology induced from $U(R)$ and $V(R)$ match the topology given by viewing these distinguished opens as affine varieties; thus the topologies match on the overlap, so these match we have defined a topology on $X(R)$. It is straightforward to check properties $(1)-(5)$.

Now we consider a more general class of rings. Let $I$ be an index set, and for each $i\in I$ let $R_i$ be a continuously invertible topological local ring. Let $R=\prod_{i\in I}R_i$ be the product. We will describe how to give a good topologization on the $R$-points of finite-type $R$-schemes for this class of $R$. For any cofinite subset $J\subseteq I$, we set $R_J=\prod_{i\in J}R_i$. Then $R=(\prod_{i\in I\setminus J}R_i)\times R_J$.
 
As $R_i$ is a continuously invertible topological local ring, we have already described a topologization of the $R_i$-points of any finite-type $R_i$-scheme. Let $X$ be a finite-type $R$-scheme. For any cofinite $J\subseteq I$, and open $U\subseteq \prod_{i\in I\setminus J} X(R_i)$, consider $U\times X(R_J)\subseteq (\prod_{i\in I\setminus J} X(R_i))\times X(R_J)=X(R)$; we define the topology on $X(R)$ to have these subsets as a basis.

We check properties $(1)$ and $(2)$. First if $X=\A^1_R$, then $X(R)=R$ as sets. A basis for the topology is given by sets of the form $U\times R_J$ for cofinite subsets $J\subseteq I$ and open $U\subseteq \prod_{i\in I\setminus J}\A^1(R_i)$; this is precisely the topology on $R$ as a direct product so $(1)$ is established.

Let $f:X\to Y$ be a morphism of finite-type $R$-schemes. Let $J\subseteq I$ a cofinite subset and let $U\subseteq \prod_{i\in I\setminus J}Y(R_i)$ be an open . Let $g:\prod_{i\in I\setminus J}X(R_i)\to \prod_{i\in I\setminus J}Y(R_i)$ be induced from $f$; this is continuous as each $R_i$ is continuously invertible and local and we have checked the topologization is excellent for these rings. Therefore, $g^{-1}(U)$ is open in $\prod_{i\in I\setminus J}X(R_i)$. We then have that $f^{-1}(U\times Y(R_J))= g^{-1}(U)\times X(R_J)$ and this is open in $X(R)$ as it is in the described basis of opens. Therefore $X(R)\to Y(R) $ is continuous, so we have $(2)$.

To check property $(3)$, let $X\to Y$ be a closed immersion of finite-type $R$-schemes. We first check that image is a closed set. The image is the intersection over $j\in I$ of $X(R_j)\times\prod_{i\in I, i\neq j}Y(R_i)$, which are each closed, so the image is closed. Now we check that $X(R)\to Y(R)$ is a topological immersion. For any $i$, $X(R_i)\to Y(R_i)$ is a closed immersion, thus a subset of $X(R_i)$ is open if and only if it is the restriction of an open subset of $Y(R_i)$. For any cofinite $J$ and open subsets $V_i\subseteq Y(R_i)$ for $i\in I \setminus J$, the subset $( \prod_{i\in I\setminus J}V_i)\times Y(R_J)\subseteq Y(R)$ intersected with $X(R)$ is $( \prod_{i\in I\setminus J}V_i\cap X(R_i))\times X(R_J)\subseteq Y(R)$. As opens of the form $V_i\cap X(R_i)$ are a basis for the topology on $X(R_i)$, opens of the form $( \prod_{i\in I\setminus J}V_i\cap X(R_i))\times X(R_J)$ are a basis of the topology on $X(R)$. We conclude that $X(R)\to Y(R)$ is an embedding of topological spaces. 

Lastly we check property $(4)$. Let $X$ and $Y$ be two finite-type $R$-schemes. A basis of opens for the product $X(R)\times Y(R)$ may be constructed as follows: choose a cofinite $J\supseteq I$, and for each $i\in I\setminus J$ choose opens $V_i\in X(R_i)$ and $W_i \in Y(R_i)$, then take the open of $X(R)\times Y(R)$ given by $\prod_{i\in I\setminus J}(V_i\times W_i)\times \prod_{i\in J}(X(R_i)\times Y(R_i))$. Now  note that the opens of the form $V_i\times W_i$ are a basis fo $(X\times Y)(R_i)=X(R_i)\times Y(R_i)$. Thus identifying $(X\times_R Y)(R)$ and $X(R)\times Y(R)$, the product topology matches the topology on $(X\times_R Y)(R)$ as desired.

\begin{rmk}\label{rem:BC}
  In \cite{Conradadeles}[page 7], it is shown that if $R=\prod_i R_i$ where each $R_i$ is a local ring and $X$ is a finite-type, quasi-separated $R$-scheme then the natural map $X(R)\to \prod_i X(R_i)$ is bijective. In this case the topology we defined on $X(R)$ agrees with the product topology on $\prod_i X(R_i)$.

\end{rmk}

\begin{defi}
  Let $R$ a topological ring. We say a good topologization of $R$-points of finite-type $R$-schemes has the {\em smooth quotient property}  if for every smooth map $X\to Y$ of finite-type $R$-schemes such that $X(R)\to Y(R)$ is surjective, $X(R)\to Y(R)$ is a quotient.
\end{defi}

\begin{prop}\label{prop:sq}
  Let $I$ be an index set. For each $i\in I$, let $R_i$ be a continuously invertible local ring with the property that smooth maps of $R_i$-schemes induce open maps on $R_i$-points. Let $R=\prod_{i\in I}R_i$ be the restricted direct product. The good topologization of $R$-points of finite-type $R$-schemes defined above has the smooth quotient property.

  In fact if $f:X\to Y$ is a smooth morphism of finite-type $R$-schemes such that $X(R)\to Y(R)$ is surjective, then $X(R)\to Y(R)$ is an open map.
\end{prop}
\begin{proof}
  In the case when $Y(R)=\emptyset$, the result is clear so we assume that $Y(R)$ is not empty.
  
By abuse of notation, for any cofinite $J\subseteq I$, we denote by $f$ the induced map $\prod_{i\in I\setminus J}X(R_i)\to \prod_{i\in I\setminus J}Y(R_i)$.

Let $U\subseteq X(R)$ be an open subset. Let $V\subsetq U$ be an open subset of the form $W\times X(R_J)$ for some open $W\subseteq  \prod_{i\in I\setminus J}X(R_i)$. Such opens cover $U$. Now $f(W)\subseteq \prod_{i\in I\setminus J}Y(R_i)$ is open as smooth maps induce open maps on $R_i$-points by hypothesis. Therefore $f(V)=f(W)\times Y(R_J)$ is open. The set $f(U)$ is a union of these $f(V)$ and is thus open as claimed.
\end{proof}

\section{Definitions}\label{sec:def}

\begin{defi}
  Let $S$ be a scheme and $T$ be $S$-scheme. Let $n\geq 0$
  
  If $S$ is the spectrum of field $k$, then $T=\Spec R$ for some $R$, and we define the {\em degree} of $T$ over $S$ to be $\dim_k R$, the dimension of $R$ as a $k$-vector space.

  For general scheme $S$, and for any $n\geq 0$, we say that $T\to S$ has {\em degree less than or equal to $n$} if for each $s\in S$, if $k(s)$ is the residue field at $s$, then $T\times_S \Spec k(s)$ has degree degree less than or equal to $n$.

  If $R\to R'$ is a morphism of rings such that $\Spec R'\to \Spec R$ is quasi-finite, we say that $R\to R'$ has {\em degree less than or equal to $n$} if $\Spec R'\to \Spec R$ has degree less than or equal to $n$.
\end{defi}

\begin{rmk}
  If $T\to S$ is a quasi-finite and representable map of algebraic stacks, then for every field $k$ with map $\Spec k\to S$, the pullback  $T\times_S\Spec k$ is a scheme. In this we can extend the definition of having degree less than or equal to $n$ to this situation.

  The map $T\to S$ is of degree less than or equal to $n$ if for all fields $k$ and all morphisms $\Spec k \to S$, $T\times_S \Spec k\to \Spec k$ is of degree less than or equal to $n$.
\end{rmk}

\begin{defi}\label{def:suffdisc}
  Let $R$ be a ring. We say that $R$ is {\em sufficiently disconnected}

  \begin{enumerate}[(1)]
  \item Finitely generated projective modules over $R$ are free.
  \item For any faithfully flat \'etale $R\to R'$ , there exists an $R'$-algebra $R''$ such that $R''$ is finite \'etale over $R$.
  \end{enumerate}
\end{defi}

\begin{lem}\label{lem:suffdisc}
  Let $R$ be a sufficiently disconnected ring. For any $n\geq 0$ and any faithfully flat \'etale $R\to R'$ of degree less than or equal to $n$, there exists an $R'$-algebra $R''$ such that $R''$ is finite \'etale and free over $R$ of rank $n!$.
\end{lem}
\begin{proof}
  Given $R\to R'$, let $S$ be finite \'etale $R'$-algebra given by Definition \ref{def:suffdisc}.

  Let $T\subseteq S$ be the image of $R'$. Applying Lemma \ref{lem:22223}, we conclude that $T$ is \'etale over $R$. Since $R'$ and $S$ are \'etale over $R$ they are finitely presented over $R$, so there is a noetherian subring over which $R'$ and $S$ are defined; $T$ is then defined too over that noetherian subring, and since it is a subalgebra of the module finite $S$, we conclude that $T$ too is finitely presented as an $R$-module. As $T$ is both finitely presented and flat over $R$ as an $R$-module, $T$ is a projective $R$-module. Using property $(1)$ of sufficiently disconnected, we see that $T$ is moreover free as an $R$-module. Since $T$ is a quotient of $R'$ it must have rank $m$ less than or equal to $n$. Let $R''=T^{n!/m}$. Then the diagonal defines an $R$-algebra homomorphism $T\to R''$ and the composite $R'\to T\to R''$ makes $R''$ an $R'$-algebra. Finally, by construction $R''$ is finite \'etale and free of rank $n!$ over $R$.  
\end{proof}

The following is a scheme-theoretic version of Lemma~\ref{lem:suffdisc}.
\begin{lem}\label{rem:suffdisc}  
  Let $R$ be a sufficiently disconnected ring. Let $U\to V$ be a separated, surjective \'etale map of algebraic stacks over $R$ which is representable and of degree less than or equal to $n$. Then for any $v\in V(R)$, there is a finite \'etale $R$-algebra $S$ of rank $n!$ and $u\in U(S)$ such that $u$ and $v$ have the same image in $V(S)$.
\end{lem}
\begin{proof}

  Consider the pullback $U\times_{V,v}\Spec R\to \Spec R$. This is a separated algebraic space that surjective \'etale over $\Spec R$. If we find a finite \'etale $R$-algebra $S$ of rank $n!$ and $u'\in (U\times_{V,v}\Spec R)(S)$, then if we denote the composite $\Spec S\xrightarrow{u'} U\times_{V,v}\Spec R\to U$ by $u$, $u$ has the desired properties. In this way we reduce to the case when $U$ is a separated algebraic space and $V=\Spec R$.

  Now let $W\to U$ be an \'etale cover by an affine scheme. Then $W=\Spec R'$ for some \'etale $R$-algebra $R'$. Applying Definition \ref{def:suffdisc} gives an $R'$-algebra $R''$ which is finite \'etale over $R$. By Remark \ref{rem:22223}, the image of $\Spec R''\to U$ is finite \'etale over $R$, and thus an affine $R$-scheme. This image is of the  form $\Spec S'$ for some $R$-algebra $S'$. We have a map $u':\Spec S'\to U$ by construction. Now we may take a finite partition of $\Spec R$ by open sets such that over each open set in the partition $\Spec S'$ has constant degree. It suffices to prove the result separately over each open in the partition. Thus passing to an open set in this partition, we may assume that $S'$ is finite \'etale of some fixed rank $m$ over $R$. Set $S=(S')^{n!/m}$. The diagonal defines a map $S'\to S$, we have a composite map $\Spec S\to \Spec S'\xrightarrow{u'}U$ which we denote by $u$. Then $S$ and $u$ have the desired properties.
\end{proof}

\begin{prop}\label{prop:x8x8}
  Complete noetherian local rings are sufficiently disconnected.
\end{prop}

Before we prove this, we prove an intermediary lemma.

\begin{lem}\label{8989}
  Let $R$ be a complete noetherian local ring. Let $R'$ be an \'etale $R$-algebra. Then $R'$ is of the form $T\times S$ where $T$ is finite \'etale over $R$ and $S$ is  \'etale but not faithfully flat over $R$.
\end{lem}
\begin{proof}
  Let $\m$ be the maximal ideal of $R$. Let $\widehat{R'}$ be the $\m$-adic completion of $R'$.

  Because $R\to R'$ is \'etale, $R'/\m^eR'$ is \'etale over $R/\m^e R$. Furthermore, $R'/\m^eR'$ is finite over $R/\m^eR$ as it is quasi-finite over an artinian local ring. Thus, we have that $\widehat{R'}$ is finite \'etale over $R=\widehat{R}$. Set $T=\widehat{R'}$.

  Now as $R'\to R'/\m R'$ is surjective, by Nakayama's lemma we conclude that $R'\to T$ is surjective. By Lemma \ref{lem:22223}, $\Spec \widehat{R}'\to \Spec R'$ has clopen image. Therefore,  we may write $R'=\widehat{R'}\times S$ for some (necessarily \'etale) $R$-algebra $S$. By construction $R'/mR'\to T/\m T$ is surjective, so $S/mS =0$, so as claimed $S$ is not  faithfully flat over $R$.  
\end{proof}

\begin{proof}[Proof of Proposition \ref{prop:x8x8}]
  Let $R$ be a complete noetherian local ring with maximal ideal $\m$. Property $(1)$ of Definition \ref{def:suffdisc} is clear, so we prove property $(2)$.
  
  Let $R'$ be an \'etale faithfully flat $R$-algebra.  We will find an $R'$-algebra $R''$ such that $R\to R''$ is finite \'etale.
 
  By Lemma \ref{8989}, $R'=T\times S$ for $T$ finite \'etale over $R$. We take $R''=T$.  
\end{proof}

\begin{defi}\label{def:essan}
  Let $R$ be a topological ring. We say $R$ is {\em essentially analytic} if $R$ is a local, sufficiently disconnected, continuously invertible topological ring such that every \'etale map $X\to Y$ of finite-type $R$-schemes induces a local homeomorphism $X(R)\to Y(R)$.
\end{defi}
\begin{rmk}
  Essentially analytic rings include $\Z_p$, $\Q_p$, $\R$, and $\mathbb{C}$.

  For these rings, the fact that they are essentially analytic comes from the inverse function theorem.
\end{rmk}

Here is the main definition of the paper.
\begin{defi}\label{def:top}
  Let $R$ be an essentially disconnected, continuously invertible topological ring. Let $\Xf$ be an algebraic $R$-stack. We topologize $\Xf(R)$ in the following way:

  For each smooth cover $Z\to \Xf$ from an $R$-scheme $Z$, let $\Xf(R)_Z$ be the image of $Z(R)$ in $\Xf(R)$ given the quotient topology.

  Let $\C_\Xf$ be the category of $R$-schemes $Z$ together with a smooth cover $Z\to \Xf$, and maps from $Z_1\to \Xf$ to $Z_2\to \Xf$ are smooth maps $Z_1\to Z_2$ making the following diagram commute
  \[
    \begin{tikzcd}
      Z_1 \arrow[rr]\arrow[dr] & & Z_2\arrow{dl}\\
      & \Xf&
    \end{tikzcd}
  \]

  Then let $\Xf(R)$ be given the topology of $\colim_{Z\in \C}\Xf(R)_Z$.
\end{defi}

For this definition to make sense, we will need to see that every $x\in \X(R)$ is in $\X(R)_Z$ for some $Z$. Theorem \ref{thm:lift} establishes this.

\begin{rmk}\label{rem:filtered}
  Note that in Definition \ref{def:top}, we get the same colimit if we restrict to morphisms of the form $Z_1\hookrightarrow Z_1\coprod Z_2$, for smooth cover $Z_2\to \Xf$. In this way, if convenient, we may assume that the colimit is over a filtered category.
  
\end{rmk}

\section{Cocycle spaces}\label{sec:cocycle}
In this section $R$ is a ring and all objects are finite-type over $R$. Additionally through out this section, $T$ will be an $R$-scheme, $T'\to T$ will be an fppf cover, $\Xf$ will be an algebraic stack over $T$, and $Z$ an algebraic space over $T$, and $Z\to \Xf$ will be a smooth cover. We briefly recall Section \ref{sec:stackssetup}.

Let us be given maps $T'\to Z$ and $T'\times_TT'\to Z\times_\Xf Z$ making the following diagram commute:

  \begin{equation}\label{eqn:first}
    \begin{tikzcd}
      T'\times_TT'\ar[xshift=3pt]{d} \ar[xshift=-3pt]{d}\arrow{r}&  Z\times_{\Xf}Z \ar[xshift=3pt]{d} \ar[xshift=-3pt]{d}\\
      T'\arrow{r}& Z.
    \end{tikzcd}
  \end{equation}

 From this data we get a map $T'\times_T T'\times_T T'\to Z\times_{\Xf}Z\times_{\Xf} Z$

If the diagram 
\begin{equation}\label{eqn:2}
    \begin{tikzcd}
    T'\times_T T'\times_TT'\ar[d,"\pi_{23}"]\ar[r]& Z\times_{\Xf}Z\times_{\Xf}Z\ar[d,"\pi_{23}"]\\
    T'\times_TT'\arrow{r}&  Z\times_{\Xf}Z
  \end{tikzcd}
\end{equation}
commutes, we get an associated map $T\to \Xf$.

Now a map $T'\to Z$ is the same data as a $T$-point of $\Res_{T'/T}Z$. A map $T'\times_T T'\to Z\times_{\Xf} Z$ is the same as the data of a $T$-point of $\Res_{T'\times_T T'/T}(Z\times_\Xf Z)$.  There are two maps $\pi_1,\pi_2:\Res_{T'/T}Z\to \Res_{T'\times_TT'/T}(Z\times_{\Xf}Z)$ coming from the two projections $T'\times_T T'\to T'$. There are also two maps $\pi_1,\pi_2:\Res_{T'\times_T T'}Z\times_\Xf Z\to \Res_{T'\times T'/T}Z$ coming respectively from the two projections $Z\times_\Xf Z\to Z$.

Finally, the data of $T'\to Z$ and $T'\times T'\to Z\times_\Xf Z$ making Diagram \ref{eqn:first} commute can thus be rephrased as the data of a $T$-point of the limit, $P_{T',Z\to \Xf}$, of the following diagram:
  
\begin{equation}\label{eqn:3}
  \begin{tikzcd}
    & \Res_{T'\times_T T'}Z\times_\Xf Z\arrow{d}{\pi_1}\arrow[bend left=30]{ddr}{\pi_2}& \\
    \Res_{T'/T} Z\arrow{r}{\pi_1}\ar[bend right=20]{drr}{\pi_2}& \Res_{T'\times_T T'/T}Z & \\
    & & \Res_{T'\times_T T'/T}Z      
  \end{tikzcd}
\end{equation}

We get maps $d_u,d_\ell:P_{T',Z\to \Xf}\to \Res_{T'\times_T T'\times_T T'/T}(Z\times_\Xf Z)$ corresponding to the upper and lower paths of Diagram \ref{eqn:2}. Let $\C_{T',Z\to \Xf}$ be the equalizer of
\[
  \begin{tikzcd}
    P_{T',Z\to \Xf}\ar[yshift=3pt]{r}{d_u}\ar[yshift=-3pt,swap]{r}{d_\ell}&\Res_{T'\times_T T'\times_T T'}Z\times_\Xf Z.
  \end{tikzcd}
\]

As it stands this $\C_{T',Z\to \Xf}$ is only a sheaf. If it is representable by a scheme or algebraic space, we will refer to $\C_{T',Z\to \Xf}$ as a {\em cocycle space}.

Descent gives us a map $\C_{T',Z\to \Xf}\to \Xf$ (see Section \ref{sec:stackssetup}).

\begin{prop}\label{prop:pullback}
  Let $\Yf\to \Xf$ be a map of algebraic stacks over $T$. If $Z\times_\Xf \Yf$ is an algebraic space, then there is a canonical isomorphism $\C_{T',Z\to \Xf}\times_{\Xf}\Yf\cong \C_{T',(Z\times_\Xf \Yf)\to \Yf}$.
\end{prop}
The reason for assuming that $Z\times_{\Xf}\Yf$ is an algebraic space is that this is the only context in which we have a definition of $\C_{T',(Z\times_\Xf \Yf)\to \Yf}$. 

\begin{proof}
  We will describe a map in both directions. As the construction of the cocycle space is functorial in $T$ it will suffice to give what this map does on $T$-points.
  
  First let us begin with constructing the map $\C_{T',Z\to \Xf}\times_{\Xf}\Yf\to \C_{T',(Z\times_\Xf \Yf)\to \Yf}$. We will describe how a $T$-point of $\C_{T',Z\to \Xf}\times_{\Xf}\Yf$ gives a $T$-point of $\C_{T',(Z\times_\Xf \Yf)\to \Yf}$.

  The data of a $T$-point of $\C_{T',Z\to \Xf}\times_{\Xf}\Yf$ is the same as the following data:
  \begin{enumerate}[(1)]
  \item Maps $T'\to Z$ and $T'\times_T T'\to Z\times_{\Xf} Z$ making Diagrams \ref{eqn:first} and \ref{eqn:2} commute.

    This then defines a map $f_1:T\to \Xf$ making the following diagram commute:
    \[
      \begin{tikzcd}
        T'\ar[r]\ar[d]& Z\ar[d]\\
        T\ar[r,"f_1"]& \Xf.
      \end{tikzcd}
    \]

  \item A map $T\to \Yf$.

    This defines a composite $T\to \Yf\to \Xf$, which we denote by $f_2:T\to \Xf$.

  \item An isomorphism $f_1\cong f_2$.    
  \end{enumerate}

  The isomorphism $f_1\cong f_2$ by functoriality defines an isomorphism between the composite $T'\to Z\to \Xf$ and the composite $T'\to T\to \Yf\to \Xf$. The map $T'\to Z$, the composite $T'\to T\to \Yf$, and this isomorphism together define a map $T'\to Z\times_{\Xf} \Yf$.

  Similarly, the isomorphism $f_1\cong f_2$ by functoriality defines an isomorphism between the composite $T'\times_TT'\to Z\times_{\Xf}Z\to \Xf$ and the composite $T'\times_T T'\to T\to \Yf\to \Xf$. The map $T'\times_TT'\to Z\times_{\Xf}Z$, the composite $T'\times_T T'\to T\to \Yf$, and this isomorphism define a map $T'\times_T T'\to (Z\times_\Xf Z)\times_{\Xf}\Yf$. But $(Z\times_\Xf Z)\times_{\Xf}\Yf$ is canonically isomorphic to $(Z\times_\Xf \Yf)\times_{\Yf} (Z\times_\Xf \Yf)$. Thus we get a map $T'\times_T T'\to (Z\times_\Xf \Yf)\times_{\Yf} (Z\times_\Xf \Yf)$.

  The commutativity of Diagram \ref{eqn:2} implies that the analogous diagram produced by the maps $T'\to Z\times_{\Xf} \Yf$ and $T'\times_T T'\to (Z\times_\Xf \Yf)\times_{\Yf} (Z\times_\Xf \Yf)$ also commutes. This is   precisely the data needed to define a map $T\to \C_{T',(Z\times_\Xf \Yf)\to \Yf}$. So we have produced the map $\C_{T',Z\to \Xf}\times_{\Xf}\Yf\to \C_{T',(Z\times_\Xf \Yf)\to \Yf}$.

  We now describe the map $\C_{T',(Z\times_\Xf \Yf)\to \Yf}\to \C_{T',Z\to \Xf}\times_{\Xf}\Yf$. Again, we explain how a $T$-point of $\C_{T',(Z\times_\Xf \Yf)\to \Yf}$ gives a $T$-point of $\C_{T',Z\to \Xf}\times_{\Xf}\Yf$.

  A $T$-point of $\C_{T',(Z\times_\Xf \Yf)\to \Yf}$ is the following data:
  \begin{enumerate}[(1)]
  \item Map $T'\to Z\times_\Xf \Yf$
  \item Map $T'\times T'\to (Z\times_{\Xf}\Yf)\times_{\Yf}(Z\times_{\Xf}\Yf)\cong (Z\times_\Xf Z)\times_{\Xf}\Yf$,
  \end{enumerate} making the analogue of Diagram \ref{eqn:first} and Diagram \ref{eqn:2} commute.

  By composing with the appropriate projections, we get: \begin{enumerate}[(1)]
  \item Map $T'\to Z$
  \item Map $T'\times_T T'\to Z\times_{\Xf} Z$
  \end{enumerate} making Diagram \ref{eqn:first} and Diagram \ref{eqn:2} commute. This defines a map $T\to \C_{T',Z\to \Xf}$. Let $f_1$ be the composite $T\to \C_{T',Z\to \Xf}\to \Xf$. We can then identify the composite $T'\to Z\times_\Xf \Yf\to Z\to \Xf$ with $T'\to T\xrightarrow{f_1}\Xf$ and the composite $T'\times_T T'\to (Z\times_{\Xf}\Yf)\times_{\Yf}(Z\times_{\Xf}\Yf)\to Z\times_{\Xf}Z\to \Xf$ with $T'\times_T T'\to T\xrightarrow{f_1} \Xf$.

  Now form the compoisition $T'\to Z\times_{\Xf} \Yf\to \Yf$. The fact that the $Z\times_\Xf \Yf\to \Yf$ version of Diagram \ref{eqn:first} commutes defines an isomorphism between the two composites $T'\times_T T'\xrightarrow{\pi_i}T'\to \Yf$. The fact that the $Z\times_\Xf \Yf\to \Yf$ version of Diagram \ref{eqn:2} commutes defines a ``cocycle'' compatibility between the induced isomorphisms between three composites $T'\times_T T'\times_T T'\xrightarrow{\pi_i}T'\to \Yf$. Descent from this defines a map $f_2:T\to \Yf$.

  Furthermore, descent theory gives us canonical isomorphisms between the composites $T'\to Z\times_\Xf \Yf\to \Yf$ and $T'\to T\xrightarrow{f_2} \Yf$, and also between the composites $T'\times_T T'\to (Z\times_\Xf \Yf)\times_{\Yf}(Z\times_\Xf \Y)\to \Yf$ and $T'\times_T T'\to T\xrightarrow{f_2} \Yf$. Thus we can and do assume that these isomorphic composites are equal. (Also note that since $Z\times_\Xf \Yf$ is an algebraic space, there is at most one isomorphism between any two maps to $Z\times_\Xf \Yf$.)

  Next, the map $T'\to Z\times_{\Xf}\Yf$ gives an isomorphism between the composites $T'\to T\xrightarrow{f_1} \Xf$ and $T'\to T\xrightarrow{f_2} \Xf$. Likewise, the map $T'\times_T T'\to (Z\times_{\Xf}\Yf)\times_{\Yf}(Z\times_{\Xf}\Yf)\cong (Z\times_{\Xf} Z)\times_{\Xf} \Yf$ gives an isomorphism between the composites $T'\times_T T'\to T\xrightarrow{f_1} \Xf$ and $T'\times_T T'\to T\xrightarrow{f_2} \Xf$.

  The commutativity of the $Z\times_{\Xf}\Yf\to \Yf$ version of Diagram \ref{eqn:first} implies that isomorphism between the composites $T'\times_T T'\to T\xrightarrow{f_1} \Xf$ and $T'\times_T T'\to T\xrightarrow{f_2} \Xf$ is induced from an isomorphism of the composites $T'\to T\xrightarrow{f_1} \Xf$ and $T'\to T\xrightarrow{f_2} \Xf$ by functoriality from either projection. Then descent (or the stack property) implies that these identifications must be induced by functoriality from an isomorphism $f_1\cong f_2$.

  The data of $T\to \C_{T',Z\to \Xf}$ and $T\to \Yf$ and isomorphism $f_1\cong f_2$, together define a map $T\to \C_{T',Z\to \Xf}\times_\Xf \Yf$. Thus we have described the map $\C_{T',Z\times_\Xf \Yf\to \Yf}\to \C_{T',Z\to \Xf}\times_{\Xf}\Yf$.

  These maps are inverses.

\end{proof}

\section{Representability of cocycle spaces}

In this section $R$ is a ring and all objects are finite-type over $R$. Additionally through out this section, $T$ will be an $R$-scheme, but now $T'\to T$ will be a finite \'etale  cover. Still $\Xf$ will be an algebraic stack over $T$, and $Z$ an algebraic space over $T$, and $Z\to \Xf$ will be a smooth cover.

\begin{lem}\label{lem:12}
  Assume $\Xf$ is an algebraic space, $Z$ is relatively affine over $T$. For any $t\in T$, if $k(t)$ denotes the residue field at $t$, any finite set of points in $(Z\times_\Xf Z)\times_{T}\Spec k(t)$ is contained in an affine subset of $(Z\times_\Xf Z)\times_{T}\Spec k(t)$.
\end{lem}
\begin{proof}
  By base changing $Z\to T$ to $\Spec k(t)$, we may reduce to the case when $T$ is the spectrum of a field and $t\in T$ is the unique point and $Z$ is affine. Thus we have to show that $Z\times_X Z$ has the property that any finite collection of points is contained in an affine.

  The diagonal $\Xf\xrightarrow{\Delta} \Xf\times_{T}\Xf$ is a locally closed immersion, which implies that $Z\times_\Xf Z$ is locally closed in $Z\times_{T} Z$. On the other hand $Z$ is affine, so $Z\times_{T} Z$ is affine, and therefore that $Z\times_\Xf Z$ is quasi-affine over a field. Finite-type quasi-affine schemes over fields have the property that any finite collection of points is contained in an affine, as quasi-projective varieties over fields have this property and quasi-affine varieties are quasi-projective.
\end{proof}

\begin{prop}\label{prop:repscheme}
  If $\Xf$ is an algebraic space and $Z$ is a relatively affine $T$-scheme, Then $\C_{T',Z\to \Xf}$ is representable by a scheme.
\end{prop}
\begin{proof}
  Note that $T'\times_T T'$ and $T'\times_T T'\times_T T'$ are also both finite and locally free over $T$. The sheaf $\C_{T',Z\to \Xf}$ is constructed as a limit of various restriction of scalars of $Z$ and $Z\times_{\Xf} Z$. These restrictions of scalars are schemes by Lemma \ref{lem:12} and \cite{NeronModels}[Theorem 7.6.4]. Limits of schemes are schemes, so $\C_{T',Z\to \Xf}$ is a scheme.
\end{proof}

\begin{lem}\label{lem:12312}
  Recall we assumed $T'\to T$ is finite \'etale. Let $X'$ be an algebraic space over $T'$. Then $\Res_{T'/T}X'$ is an algebraic space over $T$.
\end{lem}
\begin{proof}
  The question is local on $T$, so we can assume that $T'\to T$ is free of a fixed rank $n$. There exists a finite \'etale cover $T''\to T$ such that $T''\times_T T'\cong (T'')^n$ (the fiber product of $n$ copies of $T'$ over $T$) as a $T$-scheme. 

  Now $(\Res_{T'/T}X')\times_{T}T''\to \Res_{T'/T}X'$ is representable, surjective, and \'etale. Furthermore, $\Res_{T'/T}X'\times_{T}T''\cong (X'\times_{T'}T'')^n$ (the fiber product of $n$ copies of $X'$ over $T'$) as functor on $T$-schemes. Therefore, $\Res_{T'/T}X'$ is an \'etale quotient of an algebraic space, so is an algebraic space.
\end{proof}

\begin{prop}\label{prop:func}
  If $\Xf$ is an algebraic stack then $\C_{T',Z\to \Xf}$ is representable by an algebraic space.
\end{prop}
\begin{proof}
  By Lemma \ref{lem:12312} all spaces used to construct $\C_{T',Z\to \Xf}$ are algebraic spaces. As the limit of algebraic spaces are algebraic spaces, $\C_{T',Z\to \Xf}$ is an algebraic space.
\end{proof}

\begin{prop}\label{prop:scheme}
  Suppose that $T'\to T$ is finite \'etale. If $\Xf$ is a scheme, then the following diagram is cartesian:
  \[
    \begin{tikzcd}
      \C_{T',Z\to \Xf}\ar{r}\ar{d}& \Res_{T'/T}Z\ar{d}\\
      \Xf\ar{r}& \Res_{T'/T} \Xf.
    \end{tikzcd}
  \]
\end{prop}
\begin{proof}
  In the case when $\Xf$ is a scheme and we have a map $T'\to Z$, there is at most one map $T'\times_T T'\to Z\times_\Xf Z$ making Diagram \ref{eqn:first} commute. Additionally, once we find such a map Diagram \ref{eqn:2} automatically commutes.

  Furthermore, we can find a map $T'\times_T T'\to Z\times_\Xf Z$ making Diagram \ref{eqn:first} commute precisely there is a morphism $T\to \Xf$ making
  \[\begin{tikzcd}
      T'\ar[r]\ar[d]&  Z\ar[d]& \\
      T\ar[r]& \Xf
  \end{tikzcd}\]
commute.

Thus a point of $\C_{T',Z\to \Xf}$ is a $T$-point of $\Xf$ and a $T'$-point of $Z$ having the same image in $\Xf(T')$. This precisely says the desired diagram is cartesian.
\end{proof}

\begin{prop}\label{prop:smooth}
  Suppose that $T'\to T$ is finite \'etale. Then the map $\C_{T',Z\to \Xf}\to \Xf$ is a smooth cover.
\end{prop}
\begin{proof}
  First suppose that $\Xf$ is a scheme. Proposition \ref{prop:scheme} expresses $\C_{T',Z\to \Xf}\to \Xf$ as a pullback of $\Res_{T'/T}Z\to \Res_{T'/T}X$. As $Z\to \Xf$ is a smooth cover, so is $\Res_{T'/T}Z\to \Res_{T'/T}\Xf$ is a smooth cover, and therefore $\C_{T',Z\to \Xf}\to \Xf$ is a smooth cover.

  In general, we pullback $\C_{T',Z\to \Xf}\to \Xf$ along a smooth cover $X\to \Xf$ from scheme $X$. Proposition \ref{prop:scheme} says that $\C_{T',Z\to \Xf}\times_\Xf X\to X$ is isomorphic to $\C_{T',Z\times_\Xf X\to X}\to X$ which is a smooth cover by the scheme case. Thus we conclude that $\C_{T',Z\to \Xf}\to \Xf$ is a smooth cover.
\end{proof}

The following is proved via a different method by M. Bhargava and B. Poonen in the case when $R$ is a noetherian local ring \cite{BP}
\begin{thm}\label{thm:lift}
  Assume $R$ be a sufficiently disconnected topological ring, and let $T=\Spec R$. Let $\Xf$ be an algebraic stack over $T$, and let $T\to \Xf$ be a section. Then there is a $T$-scheme $Y$ with smooth cover $Y\to \Xf$ and map $T\to Y$ making the following diagram commute:
  \[
    \begin{tikzcd}
      T\ar{rr}\ar{dr}& & Y\ar{dl}\\
      & \Xf. &
    \end{tikzcd}
  \]
\end{thm}
\begin{proof}
  Let $X\to \Xf$ be any separated smooth cover from by scheme . Let $Z=X\times_{\Xf} T$.
  
  Now $Z\to T$ is smooth, so there is an \'etale surjective morphism $T_1\to T$ such that $Z\times_T T_1\to T_1$ has a section. As $R$ is sufficiently disconnected by Lemma \ref{rem:suffdisc} this implies there exists $T'\to T_1$ such that the composite $T'\to T_1\to T$ is surjective finite \'{e}tale map. The induced map $T'\to T_1\to Z$ makes $T'\to Z$ making the following diagram commute
  \[
    \begin{tikzcd}
      T'\ar{rr}\ar{dr}& & Z\ar{dl}\\
      & T &
    \end{tikzcd}
  \]

  Now from the maps $T\to \Xf$ and $T'\to Z$, we can construct $T'\times_T T'\to Z\times_\Xf Z$ (which component-wise is the map $T'\to Z$), and this clearly makes Diagram \ref{eqn:first} and Diagram \ref{eqn:2} commute. This gives us a map $T\to \C_{T',Z\to X}$. Now if $\C_{T',Z\to \Xf}$ is only an algebraic space, we run the same procedure again lifting $T$ to a smooth cover of $\C_{T',Z\to \X}$ by a scheme keeping in mind Proposition \ref{prop:repscheme} . This proves the proposition.  
\end{proof}

\begin{prop}\label{prop:suffconnDM}
  Let $T=\Spec R$, so $\Xf$ is a finite-type $R$-stack. There are smooth covers $\pi_N:Z_N\to \Xf$ by schemes for $N\geq 1$, such that $\bigcup_N\pi_N(Z_N(R))=\Xf(R)$ and additionally, for any sufficiently disconnected $R$-algebra $R'$, $\bigcup_N\pi_N(Z_N(R'))=\Xf(R')$

  Additionally if $\Xf$ is a Deligne-Mumford stack, there is a single smooth cover $\pi:Z\to \Xf$ be a scheme such that $\pi(Z(R))=\Xf(R)$ and futhermore for any sufficiently disconnected $R$-algebra $R'$, $\pi(Z(R'))=\Xf(R')$
\end{prop}
\begin{proof}

  Let $X\to \Xf$ be any smooth cover by a scheme. Let $E_N$ be the $R$-scheme that parametrizes \'etale $R$-algebras which are free as $R$-modules of rank $N$ and are equipped with an $R$-basis. This is a smooth scheme. Let $F_N\to E_N$ be the universal \'etale cover of rank $N$. Set $Z_N'=\C_{F_N, E_N\times_R X\to E_n\times_R \Xf}$. By Prop \ref{prop:func}, the fiber of $Z_N'\to E_N$ over a point corresponding to a free \'etale $R$-algebra $S$ equipped with a basis as an $R$-module is $\C_{S, X\to \Xf}$. By Lemma \ref{rem:suffdisc}, for $x\in \Xf(R)$ there is a finite free \'etale $R$-algebra, $S$, of some rank $N\geq 1$, such that $x$ lifts to a point of $X(S)$. Such an $x$ thus lifts to an $R$-point of $Z_N'$. By Prop \ref{prop:func}, the same property holds for any sufficiently disconnected $R$-algebra $R'$ and $x\in \Xf(R')$. The $Z_N'$ exhibit the desired properties, except they are algebraic spaces and not necessarily schemes.

  Assume for a moment we have proved the results in the second paragraph of the statement for algebraic spaces. Then we can find a cover $Z_N\to Z_N'$ such that for any sufficiently disconnected $R$-algebra $R'$, $Z_N(R')\to Z_N'(R')$ is surjective. Because $\bigcup_N Z_N'(R')=\Xf(R')$, we also have $\bigcup_N Z_N(R')=\Xf(R')$.

  We now prove the results in the second paragraph of the statement. Assume that $\Xf$ is a Deligne-Mumford stack and let $X\to \Xf$ be a separated \'etale cover  of degree less than or equal to $n$ for some $n\geq 1$ by an affine scheme. Applying Lemma \ref{rem:suffdisc} to $X\to \Xf$ for any $x\in \Xf(R)$ there must be a free \'etale $R$-algebra $S$ of rank $n!$ and $x'\in \Xf(S)$ such that $x$ and $x'$ have the same image in $\Xf(S)$. In fact the same conclusion holds for any sufficiently disconnected $R$-algebra $R'$ and any $x\in \Xf(R')$ (by applying Lemma \ref{rem:suffdisc} to $X\otimes_R  R'\to \Xf\otimes_R R'$). This implies that for any such $R'$, $Z_{n!}'(R')\to \Xf(R')$ is surjecive. If $\Xf$ is an algebraic space, then by Proposition \ref{prop:repscheme} $Z_{n!}'$ is a scheme and we take $Z=Z_{n!}'$

  If $\Xf$ is just a Deligne-Mumford stack, $Z_{n!}'$ is only an algebraic stack. By the now completed algebraic space case, we may find a cover $Z\to Z_{n!}'$ by a scheme such that for every  any sufficiently disconnected $R$-algebra $R'$, $Z(R')\to Z_{n!}'(R')$ is surjective. The composite $Z\to Z_{n!}\to \Xf$ then has the desired properties.
\end{proof}

\begin{rmk}\label{def:cospace}
  If $\Xf$ is an algebraic stack over $R$ and $\Yf$ a Deligne-Mumford stack over $R$ with a representable map $\Yf\to \Xf$. The construction in the proof of Proposition \ref{prop:suffconnDM} gives $Z_N$ that are almost functorial in the following sense: for some $N\geq 1$ and all sufficiently disconnected $R$-algebra $R'$, $(Z_N\times_{\Xf} \Yf)(R')\to R'$ is surjective.
\end{rmk}

\section{Smooth quotient property}
In this section, $R$ will be a sufficiently disconnected topological ring together a good topologization of $R$-points of finite-type $R$-schemes with the smooth quotient property.

\begin{prop}\label{12x}
  Let $X$ be an algebraic space over $R$. Let $\pi_U:U\to X$ and $\pi_V:V\to X$ be smooth covers by schemes such that the induced maps on $R$-points are surjective. Let $\pi_U(U(R))$ denote $X(R)$ with the quotient topology from $U(R)$. Similarly define $\pi_V(V(R))$. Let $f:U\to  V$ be an $X$-morphism. The natural continuous map $\pi_U(U(R))\to \pi_V(V(R))$ is a homeomorphism.
\end{prop}
\begin{proof}
  Note that as sets $\pi_U(U(R))\to \pi_V(V(R))$ is just the identity $X(R)\to X(R)$.
  
  Set $Z=U\times_X V$, and consider the diagram
  \begin{center}
  \begin{tikzcd}
      Z\ar[r]\ar[d]& U\ar[d,"\pi_U"]\\
      V\ar[r,"\pi_V"]& X.
    \end{tikzcd}
  \end{center}

  The top and left arrow are smooth morphisms of schemes, and the map $Z(R)\to U(R)$ is surjective and thus a quotient by the smooth quotient property. Therefore, $\pi_U(U(R))$ is a quotient of $Z(R)$. Similarly, $Z(R)\to V(R)$ is surjective thus also a quotient, so the composite $Z(R)\to V(R)\to \pi_V(V(R))$ is a quotient

  Therefore, $\pi_U(U(R))$ and $\pi_V(V(R))$ are both quotients of $Z(R)$: both spaces are the underlying set $X(R)$ with the topology given as a quotient of $Z(R)$ by the map $Z(R)\to X(R)$. Therefore, the topologies on both $\pi_U(U(R))$ and $\pi_V(V(R))$ match, so the natural morphism  $\pi_U(U(R))\to \pi_V(V(R))$ must be a homeomorphism.
\end{proof}

\begin{prop}\label{ref:126}
  Let $X$ be an algebraic space over $R$, and $\pi_U:U\to X$ be a smooth cover by a scheme such that $U(R)\to X(R)$ is surjective. Then when $X(R)$ is given the topology as in Definition \ref{def:top}, $U(R)\to X(R)$ is a quotient map of topological spaces.
\end{prop}

\begin{proof}

  Recall that in Definition \ref{def:top}, we define $X(R)$ as a certain colimit over the category of smooth covers of $X$ by schemes. If $V$ is any smooth cover of $X$ by a scheme, then $U\coprod V$ is another smooth cover. Thus the smooth covers of $X$ which receive an $X$-morphism from $U$ are final in this category. Therefore, $X(R)=\colim_{V\to X}\pi_V(V(R))$ (with notation as in Definition \ref{def:top}) where the colimit is now is over smooth covers of $X$ which receive an $X$-morphism from  $U$.

  By Proposition \ref{12x}, for each smooth cover $V\to X$ by a scheme with an $X$-morphism $U\to V$, the associated map $\pi_U(U(R))\to \pi_V(V(R))$ is a homeomorphism. Therefore, every such $\pi_V(V(R))$ is homemorphic to any other, so every morphism in the colimit is a homeomorphism. We conclude that $\pi_U(U(R))\to X(R)$ must therefore be a homeomorphism. As $\pi_U(U(R))$ is by definition a quotient of $U(R)$, $X(R)$ has the topology as a quotient of $U(R)$.
\end{proof}

\begin{prop}\label{12y}
  Let $f:X\to Y$ be a morphism of algebraic spaces over $R$. Then $X(R)\to Y(R)$ is continuous. Moreover, if $f$ is smooth and $X(R)\to Y(R)$ is surjective, then $X(R)\to Y(R)$ is a quotient.
\end{prop}
\begin{proof}
  Let us first do the case of general $f$. By Proposition \ref{prop:suffconnDM}, there exists a smooth cover $U\to Y$ by a scheme such that $U(R)\to Y(R)$ is surjective. Set $Z=X\times_Y U$, and let $W\to Z$ be a smooth cover by a scheme such that $W(R)\to Z(R)$ is surjective again by Proposition \ref{prop:suffconnDM}. Then $W\to X$ is a smooth cover of $X$ by a scheme such that $W(R)\to X(R)$ is surjective. By Proposition \ref{ref:126} $W(R)\to X(R)$ is a quotient map. Then by Proposition \ref{top:a}, we conclude that $X(R)\to Y(R)$ is continuous.

  Now assume $f$ is smooth and $X(R)\to Y(R)$ is surjective. This implies that $Z(R)\to U(R)$ is surjective and thus $W(R)\to U(R)$ is surjective and thus a quotient by the smooth quotient property. The smooth quotient property implies that $W(R)\to U(R)$ is a quotient. Thus Proposition \ref{top:b} implies that $X(R)\to Y(R)$ must be a quotient as desired.
\end{proof}

\section{Functorialities}\label{sec:func}
In this section, $R$ is a sufficiently disconnected topological ring together a good topologization of $R$-points of finite-type $R$-schemes with the smooth quotient property.

\begin{lem}\label{lem:dd}
  Let $X$ be an algebraic space over $R$, $\Yf$ be an algebraic stack over $R$, and $X\to \Yf$ be a map over $R$. Then there exists a smooth covering $Y\to \Yf$ by a scheme such that the canonical map $(X\times_{\Yf}Y)(R)\to X(R)$ is a surjective quotient.
\end{lem}
\begin{proof}
  It suffices to find a $Y$ such that $(X\times_{\Yf}Y)(R)\to X(R)$ is surjective on $R$-points by Proposition \ref{12y}. Remark \ref{def:cospace}, explains the existence of such a $Y$.

\end{proof}

\begin{prop}\label{prop:de}
  Let $X$ be an algebraic space over $R$, $\Yf$ be an algebraic stack over $R$, and $X\to \Yf$ be a map over $R$. Then the natural map $X(R)\to \Yf(R)$ is continuous.
\end{prop}
\begin{proof}
  We choose $\pi:Y\to \Yf$ as in Lemma \ref{lem:dd} with $Y$ a scheme. Let $\pi(Y(R))\subseteq \Yf(R)$ given the quotient topology from $Y(R)$.
  
  Consider the following diagram:
  \begin{center}
    \begin{tikzcd}
      X\times_{\Yf}Y\ar{r}\ar{d}& Y\ar{d}{\pi}\\
      X\ar{r}& \Yf.
    \end{tikzcd}
  \end{center}
  Taking $R$-points, we get a commutative diagram of sets
  \begin{center}
    \begin{tikzcd}
      (X\times_{\Yf}Y)(R)\ar{r}\ar{d}& Y(R)\ar{d}{\pi}\\
      X(R)\ar{r}& \pi(Y(R)).
    \end{tikzcd}
  \end{center}

  The right, left, and top maps are continuous. The left map is a quotient by Lemma \ref{lem:dd}. Thus Proposition \ref{top:a} implies that $X(R)\to \pi(Y(R))$ is continuous. The map $\pi(Y(R))\to \Yf(R)$ is continuous by construction of the topology on $\Yf(R)$. Therefore the composite $X(R)\to \Yf(R)$ is continuous.
\end{proof}  

\begin{thm}
  Let $\Xf$ and $\Yf$ be stacks over $R$ with a $R$-morphism $\Xf\to \Yf$, then the natural map of sets $\Xf(R)\to \Yf(R)$ is continuous.
\end{thm}
\begin{proof}
  Let $\pi:X\to \Xf$ be a smooth cover by a scheme. Let $\pi(X(R))$ denote the set $\pi(X(R))$ with the quotient topology from $ X(R)$. By Proposition \ref{prop:de}, the natural maps $X(R)\to \Yf(R)$ are continuous. Furthermore they factor through $X(R)\to \pi(X(R))$ . As $\Xf(R)=\colim_{\pi:X\to \Xf}\pi(X(R))$, by definition, we thus conclude the map  $\Xf(R)\to \Yf(R)$ is continuous.
\end{proof}

\begin{prop}\label{ref:ringfunc}
  Let $R'$ (in addition to R) be a sufficiently disconnected ring together with a good topologization of the $R'$-points on finite-type $R'$-schemes with the smooth quotient property. Let $R\to R'$ be a continuous ring homomorphism, and assume furthermore, that $X(R)\to X(R')$ is continuous for all finite-type $R$-schemes $X$. For each  algebraic stack $\Xf$ over $R$, the natural map $\Xf(R)\to \Xf_{R'}(R')$ is continuous.
\end{prop}
\begin{proof}
  Let $\pi:Z\to \Xf$ be a smooth cover by an $R$-scheme. Give $\pi(Z(R))$ the quotient topology from $Z(R)$ and $\pi(Z(R'))$ the quotient topology from $Z(R')$. There is a commutative diagram (as sets for now):
  \begin{center}
    \begin{tikzcd}
      Z(R)\ar[r]\ar[d]& Z(R')\ar[d]\\
      \pi( Z(R))\ar[r]& \pi (Z(R')).
    \end{tikzcd}
  \end{center}

  The two vertical arrows are quotient maps, and the top arrow is continuous. This implies the bottom map is continuous by Proposition \ref{top:a}

  Taking a colimit over $Z$, we conclude that the map \[\colim_{\pi:Z\to \Xf}\pi (Z(R))\to \colim_{\pi:Z'\to \Xf\times_{\Spec R}\Spec R'}\pi (Z'(R'))\] is continuous. This precisely says that $\Xf(R)\to \Xf(R')$ is continuous.
\end{proof}

\begin{prop}\label{quoquo}
  Let $X$ be an algebraic space over $R$ and and $\Xf$ an algebraic stack over $R$. Let $X\to \Xf$ be a smooth covering such that $X(R)\to \Xf(R)$ is surjective. Then $X(R)\to \Xf(R)$ is a quotient map.
\end{prop}

\begin{proof}
  Let  $\pi:U\to \Xf$ be a smooth cover by a scheme such that $U(R)\to \Xf(R)$ is surjective. At least one exists as we can can take $U$ to be a cover of $X$ such that $U(R)\to X(R)$ is surjective using Proposition \ref{prop:suffconnDM}. In any case, then $U\times_{\Xf} X\to X$ is surjective on $R$-points so is a quotient map by Proposition \ref{12y}. Similarly, $U\times_{\Xf}X\to U$ is a quotient on $R$-point. Therefore, topologizing $\Xf(R)$ as a quotient of $U$ is the same as topologizing it as a quotient of $X$. Smooth covers of $\Xf$ that are surjective on $R$-points are final among all smooth covers (as long as one exists). The definition topology on $\Xf(R)$ is as a colimit over quotients of $U(R)$ where $U\to \Xf$ is a smooth cover and by the last sentence we can restrict ourselves to covers such that $U(R)\to \Xf(R)$ is surjective. We conclude that $\Xf(R)$ must have the topology as a quotient of $X(R)$.
\end{proof}

\section{Algebraic spaces over essentially analytic rings}\label{sec:algspace}
In this section the ring $R$ is an essentially analytic topological ring.

\begin{prop}
  Let $X$ be an algebraic space over $R$, and let $\pi\colon U\to X$ be an \'{e}tale map from a scheme $U$. Equip $\pi(U(R))$ with topology as a quotient of $U(R)$. Then the map $U(R)\to \pi(U(R))$ is a local homeomorphism. In particular it is open.
\end{prop}
\begin{proof}
  Consider the following cartesian square:
  \[
    \begin{tikzcd}
      U\times_X U\ar{r}{\pi_1}\ar{d}{\pi_2}& U\ar{d}{\pi}\\
      U\ar{r}{\pi}& X.
    \end{tikzcd}
  \]

As the two arrows to $X$ are \'{e}tale, the two arrows from $U\times_X U$ are \'{e}tale.

This leads to a commutative diagram.
\[
  \begin{tikzcd}
    (U\times_{X} U)(R)\ar{r}{\pi_1}\ar{d}{\pi_2}& U(R)\ar{d}{\pi}\\
    U(R)\ar{r}{\pi}& \pi(U(R)).
    \end{tikzcd}
  \]

  We first prove that $U(R)\to \pi(U(R))$ is open. Let $V\subseteq U(R)$ be an open subset. Since $R$ is essentially analytic and each $\pi_i$ is an \'etale moprhism of schemes, they induce local homeomorphisms on $R$ points. Thus each $\pi_i$ induce open maps on $R$-points. Therefore, $\pi_2(\pi_1^{-1}(V))$ is open. Now since $(U\times_X U)(R)=U(R)\times_{X(R)}U(R)$ as sets, a set theoretic calculation shows that $\pi^{-1}(\pi(\pi_2(\pi_1^{-1}(V))))=\pi_2(\pi_1^{-1}(V))$. By the definition of the quotient topology, this implies that $\pi(\pi_2(\pi_1^{-1}(V)))$ is open, and this is equal to $\pi(V)$. Therefore, $\pi(V)$ is open in $\pi(U(R))$. Since $V$ was arbitrary, this proves that $U(R)\to \pi(U(R))$ is an open map.

  Now consider the map $U\times_X U\xrightarrow{\pi_1}U$. The map $u\mapsto (u,u)$ provides a section. As section of an \'etale morphism is an open immersion, the diagonal $\Delta\subseteq (U\times_X U)(R)$ is open.

  As $U\times_X U\to U\times_{R} U$ is a locally closed immersion, $(U\times_X U)(R)$ has the subspace topology inherited from $U(R)\times U(R)$. In particular, the topology on $(U\times_X U)(R)$ is generated by the restriction of opens of the form $V\times W$ for $V,W\subseteq U(R)$ open.
  
  Let us now check that $U(R)\to \pi(U(R))$ is a local homeomorphism at each point of $U(R)$. Let $y\in U(R)$. As the topology at $(y,y)\in (U\times_X U)(R)$ is generated by opens of the form $(W\times W)\cap (U\times_X U)(R)$ for $W\subseteq U(R)$ open and $\Delta$ is open, there exists an open $W\subseteq U(R)$ containing $y$ such that $(W\times W)\cap (U\times_X U)(R)\subseteq \Delta$.

  For any two $w,w'\in W$, if $\pi(w)=\pi(w')$, then $(w,w')\in U(R)\times_{X(R)}U(R)\subseteq U(R)\times U(R)$, but also $(w,w')\in (W\times W)\cap (U\times_X U)(R)\subseteq \Delta$, so $w=w'$, Thus the map $W\to \pi(U(R))$ is injective. We have already concluded that $W\to \pi(U(R))$ is open. Finally, as $U(R)\to \pi(U(R))$ is a quotient map, we can conclude that $W$ is an homeomorphism onto its image, concluding the proof.
\end{proof}

  \begin{prop}\label{prop:algopen}
    Let $X$ be an algebraic space over $R$, and let $\pi:Z\to X$ be a smooth cover by a scheme. Give $\pi(Z(R))$  the quotient topology from the topology of $Z(R)$. Then the map $Z(R)\to \pi(Z(R))$ is open.
\end{prop}

\begin{proof}
  Consider the following diagram:
\[
    \begin{tikzcd}
      Z\times_X Z\ar{r}{\pi_1}\ar{d}{\pi_2}& Z\ar{d}{\pi}\\
      Z\ar{r}{\pi}& X,
    \end{tikzcd}
  \]

  which leads to the diagram

  \[
    \begin{tikzcd}
      (Z\times_X Z)(R)\ar{r}{\pi_1}\ar{d}{\pi_2}& Z(R)\ar{d}{\pi}\\
      Z(R)\ar{r}{\pi}& \pi(Z(R)).
    \end{tikzcd}
  \]

  Now let $U\subseteq Z(R)$ be open. We must show that $\pi(U)\subseteq \pi(Z(R))$ is open. Now $\pi_{1}^{-1}(U)$ is open as $\pi_1$ is continuous, and the fact that $\pi_2$ is smooth implies $\pi_2(\pi_1^{-1}(U))$ is open. Let $W= \pi_2(\pi_1^{-1}(U))$, which as the property that $\pi^{-1}(\pi(W))=W$. This implies $\pi(W)$ is open, but $\pi(W)=\pi(U)$, completing the proof.
\end{proof}

\begin{prop}
  Let $X$ be an algebraic space over $R$. Let $U$ and $V$ schemes with smooth covers $\pi_U:U\to X$ and $\pi_V:V\to X$, and a $f:U\to V$ be an $X$-morphism.

  Let $\pi_U(U(R))$ be the image of $U(R)$ in $X(R)$ with the quotient topology and similarly for $\pi_V(V(R))$. Then, inclusion $\pi_U(U(R))\to \pi_V(V(R))$ is a homeomorphism onto an open subset.\end{prop}

\begin{proof}
  We will first consider the case when $V=U\sqcup W$ where $W\to X $ is also a smooth cover. The continuous injection $\pi_U(U(R))\to \pi_V(V(R))$ is open by Proposition \ref{prop:algopen}.  Therefore since open continuous injection are homeomorphisms onto open subsets, $\pi_U(U(R))\to \pi_V(V(R))$ is homeomorphism onto an open subset.

  Now consider the general case. Let the map $U\sqcup V\to X$ be denoted by $d$. The last case says that the topologies defined on $\pi_U(U(R))$ and $\pi_V(V(R))$ map are open in that defined by $d((U \sqcup V)(R))$. Thus $\pi_U(U(R))\to \pi_V(V(R))$ must be a homeomorphism onto an open subset.
\end{proof}

\begin{prop}\label{prop:openalggg}
  Let $X$ and $Y$ be two algebraic spaces over $R$ and $f:X\to Y$ be a smooth map. Then $X(R)\to Y(R)$ is open.
\end{prop}
\begin{proof}
  We will show that for any $x\in X(R)$, there is an open subset $W\subseteq X(R)$ containing $x$ such that the map $W\to Y(R)$ is open. This is clearly sufficient.

  Let $y=f(x)\in Y(R)$. Let $U$ be a scheme with an \'{e}tale cover $\pi:U\to Y$ such that $y\in \pi(U(R))$. Let $u\in U(R)$ be such that $\pi(u)=y$. Also, let $V$ be a scheme with an \'{e}tale cover $p:V\to U\times_Y X$ such that $(u,x)\in p(V(R))$. Now the image of $V(R)$ in $X$ is open and contains $x$; call this open $W$. Additionally, $V\to U$ is smooth, so the image of $V(R)$ in $U(R)$ is open. As $U\to Y$ is \'{e}tale, $U(R)\to Y(R)$ is a local homeomorphism, the image of the composite $V(R)\to U(R)\to Y(R)$ is open. The image of this composite is the same as $f(W)$ and additionally contains $y$ as $W$ contains $x$. Thus $W$ is the desired open, completing the proof. 
\end{proof}

\begin{prop}\label{lochom}
  Let $X\to Y$ be an \'{e}tale map of algebraic spaces over $R$. Then $X(R)\to Y(R)$ is a local homeomorphism.
\end{prop}
\begin{proof}
  Let $x\in X(R)$ and let $U\to X$ be an \'{e}tale map such that there is a point $u\in U(R)$ mapping to $x$. Then $U\to X$ and $U\to Y$ are \'{e}tale, so $U(R)\to X(R)$ and $U(R)\to Y(R)$ are local homeomorphisms, so $X(R)\to Y(R)$ is a local homeomorphism when restricted to the image of $U(R)$. This is an open subset and contains $x$, so $X(R)\to Y(R)$ is a local homeomorphism at $x$. As $x$ was arbitrary this completes the proof.
\end{proof}

\begin{prop}\label{prop:hausss}
  Let $X$ be a separated algebraic space over $R$. Then the topological space $X(R)$ is  Hausdorff.
\end{prop}
\begin{proof}
  To show that $X(R)$ is Hausdorff, we will show that the diagonal $\Delta(X(R))\subseteq X(R)\times X(R)$ is closed.

  Let $x,y\in X(R)$. We will show that in some open near $(x,y)\in X(R)\times X(R)$, the restriction of the diagonal map is a closed. By Proposition \ref{prop:suffconnDM}, we may find a cover $\pi:W\to X\times X$ by a scheme such that $(x,y)$ lifts to some $w\in W(R)$.

  Consider the pullback square
  \[
    \begin{tikzcd}
      Z\ar{r}{\delta}\ar{d}& W\ar{d}{\pi}\\
      X\ar{r}{\Delta}& X\times X.
    \end{tikzcd}
  \]

  By the separated hypothesis, $X\to X\times X$ is a closed immersion, so $Z\to W$ is. Therefore $\delta(Z(R))\subseteq W(R)$ is closed by Property $(3)$ of having a excellent topologization of $R$-points on finite type $R$-schemes.

  Additionally, $\pi^{-1}(\pi(\delta(Z(R))))=\delta(Z(R))$ as the above diagram gives a cartesian square of sets on taking $R$-points. As $\pi$ is open by Proposition \ref{prop:openalggg} we have that $\pi(W(R))$ is open. Additionally, as $\pi$ is open it is a quotient map onto its image. The definintion of the quotient topology gives that $\pi(\delta(Z(R))$ is closed in $\pi(W(R))$. But $\pi(\delta(Z(R)))=\Delta(X(R))\cap \pi(W(R))$. As $\pi(W(R))$ is an open containing $(x,y)$, $\Delta(X(R))$ is a closed near $(x,y)$. As $(x,y)$ is arbitrary, we conclude $\Delta(X(R))$ is immersion so $X$ is Hausdorff.
\end{proof}

\begin{prop}\label{prop:chow}
  In addition to being essentially analytic assume that $R$ is a local field or the valuation ring of a local field. Let $X\to Y$ be a proper morphism of separated alegbraic spaces over $R$. Then $X(R)\to Y(R)$ is a proper morphism of topological spaces.
\end{prop}
\begin{proof}
  Chow's lemma for algebraic spaces (\cite[\href{https://stacks.math.columbia.edu/tag/088P}{Tag 088P}]{stacks-project}) implies that there is an algebraic space $X'$ which is a closed sub-algebraic space of $\P^n_Y$ with a surjective map $X'\to X$. Now since $X'\to X$ is surjective, to prove $X\to Y$ is proper, it suffices to prove that $X'\to X$ and $X'\to Y$ are proper. Then $X'$ is also a closed sub-algebraic space of $\P^n_X$. It is thus also a closed subspace of $\P^n_Y$. It therefore suffices to prove that $\P^n_X\to X$ and $\P^n_Y\to Y$ induce proper maps on $R$-points. Thus to prove the proposition, it suffices to prove the statement: if $Y$ is an algebraic space and $X=\P^n_Y$, then $X(R)\to Y(R)$ is a proper map.

  Let $U\to Y$ be an \'etale cover by a scheme such that $U(R)\to Y(R)$ is surjective, which is possible by Proposition \ref{prop:suffconnDM}. By Proposition \ref{lochom}, $U(R)\to Y(R)$ is a local homeomorphism. So the conclusion for $\P^n_Y\to Y$ follows from that of $\P^n_U\to U$. For the $R$ as in the statemetn, $\P^n_U(R)\to U(R)$ is proper, establishing the proposition.
\end{proof}

\section{Algebraic stacks in the essentially analytic case}\label{sec:algstack}
In this this section $R$ is essentially analytic in addition to being sufficiently disconnected.

\begin{prop}\label{prop:12}
  Let $\Xf$ be an algebraic stack over $R$, and let $\pi:Z\to \Xf$ be a smooth cover by a scheme. Let $\pi(Z(R))$ denote the image of $Z(R)$ under $\pi$ viewed as a set. Topologize $\pi(Z(R))$ by giving it the quotient topology viewing $\pi(Z(R))$ as a quotient of $Z(R)$. Then the map $Z(R)\to \pi(Z(R))$ is open.
\end{prop}

\begin{proof}
  The diagram:
\[
    \begin{tikzcd}
      Z\times_\Xf Z\ar{r}{\pi_1}\ar{d}{\pi_2}& Z\ar{d}{\pi}\\
      Z\ar{r}{\pi}& \Xf
    \end{tikzcd}
  \]
  leads to the diagram of topological spaces:
  \[
    \begin{tikzcd}
      (Z\times_\Xf Z)(R)\ar{r}{\pi_1}\ar{d}{\pi_2}& Z(R)\ar{d}{\pi}\\
      Z(R)\ar{r}{\pi}& \pi(Z(R)).
    \end{tikzcd}
  \]

  Let $U\subseteq Z(R)$ be open. We must show that $\pi(U)\subseteq \pi(Z(R))$ is open. Now $\pi_{1}^{-1}(U)$ is open as $\pi_1$ is continuous, and the fact that $\pi_2$ is smooth implies $\pi_2(\pi_1^{-1}(U))$ is open by Proposition \ref{prop:algopen}. Let $W= \pi_2(\pi_1^{-1}(U))$, which as the property that $\pi^{-1}(\pi(W))=W$. This implies $\pi(W)$ is open, but $\pi(W)=\pi(U)$, completing the proof.
\end{proof}

\begin{prop}
   Let $\Xf$ be an algebraic stack over $R$. Let $\pi_X:X\to \Xf$ and $\pi_Y:Y\to \Xf$ be smooth covers by schemes $X$ and $Y$, with a map $f:X\to Y$ such that $\pi_X\cong \pi_Y\circ f$. Let $\pi_X(X(R))$ denote the image of $X(R)$ in $\Xf(R)$ (as a set) topologized with the quotient topology. Similarly, define $\pi_Y(Y(R))$. Then the inclusion $\pi_X(X(R))\to \pi_Y(Y(R))$ is a homeomorphism onto an open subset.
\end{prop}
\begin{proof}
 We will first consider the case when $Y=X\sqcup Z$ where $Z\to \Xf $ is also a smooth cover. The continuous injection $\pi_X(U(R))\to \pi_Y(V(R))$ is open by Proposition \ref{prop:12}.  Therefore since open continuous injection are homeomorphisms onto open subsets, $\pi_X(X(R))\to \pi_Y(Y(R))$ is homeomorphism onto an open subset.

  Now consider the general case. Let the map $X\sqcup Y\to X$ be denoted by $d$. The last case says that the topologies defined on $\pi_X(X(R))$ and $\pi_Y(Y(R))$ are homeomorphisms onto open subsets of $d((U \sqcup V)(R))$. Thus $\pi_X(X(R))\to \pi_Y(Y(R))$ must be a homeomorphism onto an open subset.

\end{proof}

\begin{prop}
  Let $\Xf$ be an algebraic stack over $R$. Let $\pi:Z\to \Xf$ be a smooth cover by a scheme. Let $\Xf(R)$ be given the topology as in Definition \ref{def:top}. Then $Z(R)\to \Xf(R)$ is open
\end{prop}
\begin{proof}
  This follows immediately from the previous proposition.
\end{proof}

\begin{thm}\label{thm:open}
  Let $f:\Xf\to \Yf$ be a smooth morphism of algebraic stacks over $R$. Then the induced map $f:\Xf(R)\to \Yf(R)$ is open.
\end{thm}
\begin{proof}

  Let $x\in \Xf(R)$. We can choose smooth covers $\pi:U\to \Xf$ to $\rho:V\to \Yf$ by schemes such that there exists $u\in U(R)$ mapping to $x$ and such that there is a smooth map $g:U\to V$ making the following diagram commute:

  \[
    \begin{tikzcd}
      U \ar{r}{g}\ar{d}{\pi} & V\ar{d}{\rho}\\
      \Xf\ar{r}{f} & \Yf.
    \end{tikzcd}
  \]

  The morphism $U\to \Yf\times_{\Xf}V$ is smooth, but $\Yf\times_{\Xf}V\to V$ is also smooth, so $U\to V$ is smooth thus $U(R)\to V(R)$ is an open map, thus a quotient onto its image. The sets $\pi(U(R))$ and $\rho(V(R))$ have the topologies as quotients of $U(R)$ and $V(R)$.

  Now take an open $W\subseteq \Xf(R)$ containing $u$. We may shrink $W$ so that $W\subseteq \pi(U(R))$. Let $W'=\pi^{-1}(W)\subseteq U(R))$. As we have seen all maps are open except maybe the bottom one, $\rho(g(W'))=f(W)$ is open. Thus $f$ is open.  
\end{proof}

\begin{prop}\label{prop:shaus}
  Let $R$ be a local field or the valuation ring of a local field. Let $\Xf$ be a separated algebraic stack over $R$. Then $\X(R)$ is a Hausdorff topological space.
\end{prop}

\begin{proof}
  The proof is very similar to the proof of Proposition \ref{prop:hausss}
  
  To show that $\Xf(R)$ is Hausdorff, we will show that the diagonal $\Xf(R)\subseteq \Xf(R)\times \Xf(R)$ is closed.

  Let $x,y\in \Xf(R)$. We will show that in some open near $(x,y)\in \Xf(R)\times \Xf(R)$, the restriction of the diagonal map is a closed immersion. By Theorem \ref{thm:lift}, we may find a smooth cover $\pi:W\to \Xf\times \Xf$ by a scheme such that $(x,y)$ lifts to some $w\in W(R)$.

  Consider the pullback square
  \[
    \begin{tikzcd}
      Z\ar{r}{\delta}\ar{d}& W\ar{d}{\pi}\\
      \Xf\ar{r}{\Delta}& \Xf\times \Xf.
    \end{tikzcd}
  \]

  By the separated hypothesis, $\Xf\to \Xf\times \Xf$ is a proper  morphism, so $Z\to W$ is. Therefore $\delta(Z(R))\to W(R)$ is a proper map of topological spaces by Proposition \ref{prop:chow}. 

  Additionally, it is easily checked $\pi^{-1}(\pi(\delta(Z(R))))=\delta(Z(R))$. As $\pi$ is open by Theorem \ref{thm:open} we have that $\pi(W(R))$ is open. Additionally, as $\pi$ is open it is a quotient map onto its image. The definintion of the quotient topology gives that $\pi(\delta(Z(R))$ is closed in $\pi(W(R))$. But $\pi(\delta(Z(R)))=\Delta(X(R))\cap \pi(W(R))$. As $\pi(W(R))$ is an open containing $(x,y)$, $\Delta(\Xf(R))$ is a closed near $(x,y)$. As $\Delta(\Xf(R))$ is the image of the diagonal $\Xf(R)\to \Xf(R)\times\Xf(R)$, we conclude this diagonal is closed near $(x,y)$. As $(x,y)$ is arbitrary, we conclude the diagonal in $\Xf(R)\times \Xf(R)$ is closed immersion so $\Xf$ is Haussdorff

\end{proof}

\begin{lem}\label{lem:finn}
  Let $R$ be a local field of characteristic zero. Let $\Xf\to \Yf $ be a surjective map of finite-type Deligne-Mumford stacks over $R$. Then there exists a finite extension $R'$ of $R$ such that each point in $y\in \Yf(R)$ lifts to point in $\Yf(R')$.
\end{lem}
\begin{proof}
  As there are only finitely many extensions of $R$ of a given degree, it suffices to prove that there is a number $n$ such for that each $y\in \Yf(R)$ there is a finite extension $R'$ of  $R$ of degree less than or equal to $n$, such that $y$ lifts to some point of $\Xf(R')$.

  If $\Xf\to \Yf$ is \'etale, then the degree must be bounded and the result of the lemma holds. If $X\to \Xf$ is any smooth cover, if the result of the lemma holds for $X\to \Yf$ it holds for $\Xf \to \Yf$. In this way we pass to the case when $\Xf$ is an algebraic space.

  Now if the $Y\to \Yf$ is a surjective morphism for which the result of the lemma holds, and it holds for $\Xf\times_{\Yf} Y\to Y$, then it holds for $\Xf \to \Yf$. As the result of the lemma holds for \'etale covers, we may thus replace $\Yf$ by a scheme. We have reduced to the case when $\Xf$ is an algebraic space and $\Yf$ is a scheme. But we take any surjective map $X\to \Xf$ by a scheme, and again it suffices to show the result of the lemma holds for $X\to \Yf$.

  Therefore, we must prove the result of the lemma holds for $X\to Y$ a surjective map of schemes. In this case, there is an open subset $U\subsetq Y$ over which $X\to Y$ is smooth. Smooth maps have \'etale sections.  Therefore, over the smooth locus the result of the lemma follows from the \'etale case. We then reduce to $Y\setminus U$ and lemma then follows from noetherian induction.  
\end{proof}

\begin{prop}
  Let $R$ be a local field of characteristic zero. Let $\Xf$ be a proper Deligne-Mumford stack with finite  diaganal over $R$. Then $\Xf(R)$ is a compact topological space.
\end{prop}
\begin{proof}

  By Chow's lemma (see \cite{olssonstacks}), we may find a surjective map $X\to \Xf$ where $X$ is a projective scheme over $R$.

  By Lemma \ref{lem:finn} there is a finite extension $R'$ of $R$ such that each $R$ point of $\Xf$ lifts to an $R'$-point of $X$. Then let $Z$ be the pullback in the following diagram:

  \[
    \begin{tikzcd}
      Z\ar{r}\ar{d}& \Res_{R'/R} X\ar{d}\\
      \Xf \ar{r}& \Res_{R'/R}\Xf.
    \end{tikzcd}
  \]

  As $X$ is projective, the top right is a scheme.
  
  Then $Z$ is proper over $\Xf$ as it is the pullback of a proper map. As $\Xf$ is proper, this implies that $Z$ is proper. Therefore, $Z(R)$ is compact, and as $Z(R)\to X(R)$ is surjective and $\Xf(R)$ is Hausdorff by Proposition \ref{prop:shaus}, we conclude that $X(R)$ is compact.
\end{proof}

\begin{lem}\label{ff1}
  If $R$ and $R'$ are two continuously invertible, sufficiently disconnected, topological local rings and $R\to R'$ is a continuous ring homomorphism, then for any finite-type $R$-scheme $X$, $X(R)\to X(R')$ is continuous.
\end{lem}

\begin{proof}
  As $R$ and $R'$ are both local, any $R$-point of $X$ is contained in an affine open. In this way we reduce to the case that $X$ is affine using property $(5)$ of Definition \ref{def:good}. Next, we embed $X$ in $\A^n_R$  for some $R$, and reduce to the  case when $X=\A^n_R$ using property $(3)$ of Definition \ref{def:good}. Then the fact that $\A^n_R(R)\to \A^n_R(R')$ is continuous follows from properties $(1)$ and $(4)$ of Definition \ref{def:good}.
\end{proof}

\begin{thm}\label{thm:opensubring}
  Let $R'$ be an essentially analytic topological ring. Let $R$ be an open subring of $R'$ that is also essentially analytic. Let $\Xf$ be an algebraic stack over $R$. Then $\Xf(R)\to \Xf(R')$ is an open map.

  If additionally $R$ is a complete discrete valuation ring with the usual topology given by the maximal ideal, and $R'$ is its fraction field and $\Xf$ is separated, then $\Xf(R)\to \Xf(R')$ is an open embedding.
\end{thm}
\begin{proof}
By Lemma \ref{ff1}, Proposition \ref{ref:ringfunc} applies, so $\Xf(R)\to \Xf(R')$ is continuous.

For every $N\geq 1$, let $\pi_N:Z_N\to \Xf$ be as in Proposition \ref{prop:suffconnDM}. For every $x\in \Xf(R)$ there exists an $N$ such that $x$ lifts to a point $z\in Z_N(R)$, and for every $x\in \Xf(R')$ there exists an $N$ such that $x$ lifts to a point $z\in Z_N(R')$.

  Theorem \ref{thm:open} implies that $\Xf(R)=\bigcup  \pi_{N}(Z_{N}(R))$ where each $\pi_N(Z_N(R))$ is an open subset, and also that $\Xf(R')\cong \bigcup_N\pi_{N}(Z_{N}(R'))$.

  Now consider the commutative diagram (as sets first)
  \begin{center}
    \begin{tikzcd}
      Z_{N}(R)\ar[r]\ar[d]& Z_{N}(R')\ar[d]\\
      \pi_{N}( Z_{N}(R))\ar[r]&\pi_{N}( Z_{N}(R')).
    \end{tikzcd}
  \end{center}

  The vertical arrows are quotients and are open by Theorem \ref{thm:open}, and the top arrow is open. Thus the bottom arrow must be an open. As the union of open maps is open, $\Xf(R)\to \Xf(R')$ is an open inclusion.

  If $R$ is a complete discrete valuation ring and $R'$ its fraction field and $\Xf$ is separated, then $\Xf(R)\to \Xf(R')$ is injective by the valuative criterion for separatedness. In this case, $\Xf(R)\to \Xf(R')$ is an injective open map, so must be an open immersion.
\end{proof}

\section{Products of local rings}
  Let $I$ be an index set, and let $R=\prod_{i\in I}R_i$ be a product of rings $R_i$. For any $i\in I$, let $e_i\in R$ be the element which is $1$ in the $i$th position and $0$ in all others. For $J\subseteq I$, let $e^J$ be the element of $R$ whose the $i$th component is $1$ if $i\in J$ and $0$ otherwise, and let $R^J=\prod_{i\in J}R_i=R[1/e^J]$. Let $U^J=\Spec R[1/e^J]\subseteq \Spec R$. For $J_1,J_2\subseteq I$, $U^{J_1}\cap  U^{J_2}=U^{J_1\cap J_2}$ as $\Spec R[1/e^{J_1}]\cap \Spec R[1/e^{J_2}]=\Spec R[1/e^{J_1}e^{J_2}]=\Spec R[1/e^{J_1\cap J_2}]$. Similarly, $U^{J_1}\cup  U^{J_2}=U^{J_1\cup J_2}$. In particular, this implies that that if $J_1$ and $J_2$ are disjoint, $U^{J_1}\cap U^{J_2}=\emptyset$, and if $J_1,\ldots,J_k$  is a partition of $I$, $\Spec R=U^{J_1}\coprod \ldots\coprod U^{J_k}$. We will use the notation and facts presented in this paragraph throughout the section.

\begin{prop}
  Let $I$ be an index set and $\{k_i\}_{i\in I}$ be a set of fields indexed by $I$. Let $R=\prod_{i\in I}k_i$. Then any open cover of $\Spec R$ can be refined to a disjoint open cover, and furthermore the open sets in that disjoint open cover may be taken to be of the form $\Spec R[1/e^J]\subseteq \Spec R$ for $J\subseteq I$.
\end{prop}
\begin{proof}
  Let $\{U_j\}$ be an open cover of $\Spec R$. By refining the cover, we may assume each $U_j$ equals $\Spec R[1/f_j]$ for some $f_j\in R$. As $\Spec R$ is quasi-compact, we may assume that the cover is finite.

  Now we claim each element of $R$ is the product of a unit and an idempotent. Indeed let $r\in R$ and let $r_i\in k_i$ be its $i$th component. Let $J=\{i\in I: r_i\neq 0\}$. Let $s_i=r_i$ for $i\in J$ and $s_i=1$ for $i\notin J$. Then the element of $R$ which is $s_i^{-1}$ in the $i$th position is an inverse of $s$, so $s$ is a unit. Futhermore, $r=s e^J$, so $r$ has the claimed form.

  We conclude from this that there are $J_j\subseteq I$ such that $f_j=u_j e^{J_j}$ where $u_j$ is a unit. Then $\Spec R[1/f_j]=\Spec R[1/e^{J_j}]$.

  As the $\Spec R[1/e^{J_j}]$ cover $\Spec R$, $\bigcup_j J_j =I$ .We refine the finite union $\bigcup_j J_j=I$ to a finite partition of $I$, $\coprod_{k} K_k=I$. Then $\coprod_k \Spec R[1/e^{K_k}]$ is a pairwise disjoint open cover refining the original cover.
\end{proof}

\begin{prop}\label{prop:fofo}
  Let $I$ be an index set, and let $\{\mathscr{O}_v\}_{v\in I}$ be a set of local rings indexed by $I$. Let $R=\prod_{v\in I}\mathscr{O}_v$ and let $X=\Spec R$. Then any open cover of $X$ can be refined to a disjoint open cover, and furthermore the opens in that disjoint open cover may be taken to be of the form $\Spec R[1/e^J]\subseteq \Spec R$ for $J\subseteq I$.
\end{prop} 

\begin{proof}
  For each $v$, let $\m_v$ be the maximal ideal of $\mathscr{O}_v$ and let $k_v$ be the residue field. Let $M=\{(r_v)_{v\in I}\in R: \text{ for all }r_v\in \m\}$.

  Let $\{U_j\}$ be a cover of $\Spec R$. As $\Spec R$ is quasi-compact, we can refine this cover to a finite cover. We can refine the cover further so that each $U_j= \Spec R[1/f^j]$ for $f^j\in I$.

  Let $J_j=\{v  \in I: f^j_v\in \mathscr{O}_v^\times\}$. Because $f^j|e^{J_j}$, $\Spec R[1/e^{J_j}]\subseteq \Spec R[1/f^j]$.

  Note that the image of $f^j$ in $R/M=\prod_{v\in I}k_i$ is equal to the image of $e^{J_j}$ in $R/M$ times a unit. This means that the $U_j$ cover $\Spec R/M$ if and only if the $\Spec R[1/e^{J_j}]$ do. But  by the preparation paragraph at the beginning of this section, this implies $\bigcup_j J_j =I$. Again, by that paragraph this implies that $\bigcup_j \Spec R[1/e^{J_j}]=\Spec R$, and thus $\bigcup_j \Spec R[1/e^{J_j}]$ is a refinement of our original cover.

  Now again, let $\{K_k\}$ be a finite partition of $I$ which refines $\{J^j\}$. Then $\{\Spec R[1/e^{K_k}]\}$ is a refinement of $ \{\Spec R[1/e^{J_j}]\}$, which is a refinement of our original cover. But now as the $K_k$ are disjoint, the cover by the $\Spec R[1/e^{K_k}]$ is a disjoint open cover and thus is the desired cover.
\end{proof}

\begin{lem}\label{lem:etale}
  Let $\{R_i\}_{i\in I}$ be a collection of complete discrete valuation rings and let $N\geq 1$ be an integer. For each $i$, let $S_i$ be a finite free \'etale $R_i$-algebra of rank less than or equal to $N$. Set  $R=\prod_i R_i$ and $S=\prod_i S_i$. Then $S$ is a finite \'etale $R$-algebra.
\end{lem}
\begin{proof}
  For any $i\in I$, $S_i$ is the direct product of finitely many discrete valuation rings. To $i$ we may attach a multiset $M_i$ whose elements are the ranks over $R_i$ of the discrete valuation rings whose product is $S_i$. As the rank of $S_i$ is bounded, there  are only finitely many possible such multisets. Thus we find a finite partition of $I$ such that every subset $J$ in the partition has the property that the $M_i$ are the same for all $i\in J$. As the conclusion of the lemma is local on $\Spec R$, we may thus pass to an open subset of $\Spec R$ corresponding to one subset of the partition. In this way we may assume that $M_i$ is constant for $i\in I$. Let $M=M_i$ for any $i\in I$.

  Now let $m_1,\ldots,m_k$ be an enumeration of the elements of $M$. For any $i\in I$,  $S_i$ may be written $\prod_{j=1}^kS_{i,j}$ where $S_{i,j}$ is a free \'etale $R_i$ algebra of rank $m_j$. Therefore, $\prod_i S_i=\prod_i \prod_j S_{i,j}=\prod_j\prod_i S_{i,j}$. So to prove that $S$ is finite \'etale over $R$, it suffices to show that the $\prod_i S_{i,j}$ are finite \'etale over $R$. Thus we replace $S_i$ with $S_{i,j}$ to assume that each $S_i$ is a discrete valuation of some fixed rank $m$ over $R_i$.

  Now  $S_{i}$ must be an unramified extension of $R_i$ of degree $m$, and by the principal element theorem for such extensions, we have that $S_{i}\cong R_i[x]/f_{i}(x)$ where $f_{i}(x)$ is a monic polynomial of degree $m$. Let $f(x)\in R[x]$ be the monic degree $m$ polynomial whose image in $R[x]$ is  $f_{i}(x)$ for all $i$ (its coefficients are in the $i$th component are given by the coefficients of $f_{i}$). There is a canonical map $R[x]/f(x)\to \prod_i R_i[x]/f_{i}(x)$ which is evidently both injective and surjective. Thus $S_{i}\cong R[x]/f(x)$.

  As each $S_{i}$  is \'etale, $f_i(x)$ and $f'_i(x)$ are coprime, so there are polynomials $a_i(x)$ of degree less than $m$ such that $a_i(x)f_i'(x)\cong 1$ mod $f_i(x)$. Let $a(x)\in R[x]$ be the polynomial of degree less than $m$ whose image in each $R_i[x]$ is $a_i(x)$. Then the image of $a(x)f'(x)$ in $R[x]/f(x)\cong \prod_i R_i[x]/f_{i}(x)$ is $1$. Thus $f(x)$ and $f'(x)$ are coprime, so $R[x]/f(x)$ is an \'etale $R$-algebra, and $S=R[x]/f(x)$.  
\end{proof}

\begin{lem}\label{tsts}
  Let $R$ be a product of local rings. Any finitely-generated projective module over $R$ of constant rank is free.
\end{lem}
\begin{proof}
  Let $M$ be the projective module over $R$. The module $M$ corresponds to a locally free sheaf $\widetilde{M}$ on $\Spec R$. As $M$ is projective As there is an open cover of $\Spec R$ such that the restriction of $\widetilde{M}$ is free as $\widetilde{M}$ is locally free, Proposition \ref{prop:fofo} implies there is a finite disjoint open cover such that the restriction of $\widetilde{M}$ is free when restricted to this cover. But as the cover is disjoint and the rank of the sheaf is constant, this means that $\widetilde{M}$ must be free, which implies that $M$
\end{proof}

\begin{prop}\label{prodring}
  Let $R$ be a product of fields and complete discrete valuation rings. 
\end{prop}
\begin{proof}
  Lemma \ref{tsts} already establishes property $(1)$ of sufficiently disconnected, so are left to establish property $(2)$. Let $R'$ be a faithfully flat \'etale $R$-algebra. We will find a $R'$-algebra that is finite \'etale as an $R$-algebra.
  
  For any module $M$ over $R$, we will denote by $\widetilde{M}$ its corresponding quasi-coherent sheaf on $\Spec R$.

  Write $R=\prod_{i\in I}R_i$ where $I$ is an index set and each $R_i$ is either a field or a complete discrete valuation ring. For each $i$, let $R'_i= R'\otimes_R R_i$. By Lemma \ref{8989} each $R_i'$ is of the form $S_i\times T_i$ where $S_i$ is finite \'etale over $R_i$ and $T_i$ is \'etale but not faithfully flat over $R_i$. As $R\to R'$ is faithfully flat so is $R_i\to R_i'$. Therefore, each $S_i$ is nonzero. Let $m_i$ be the degree of $S_i$ over $R_i$. Let $S=\prod_{i\in I}S_i$. Note that $S$ is finite and locally free over $R$, because it is the product of free $R_i$-modules of bounded rank. Finally, by Lemma \ref{lem:etale}, $S$ is finite \'etale over $R$. So now $S$ is free and finite \'etale over $R$ of rank $m$. Then $S$ is the desired $R'$-algebra.
\end{proof}
\begin{rmk}\label{prt}
  Let $R$ be a product of fields and complete discrete valuations. Then $R$ is sufficiently disconnected by Proposition \ref{prodring}, and the good topologization of $R$-points on finite-type $R$-schemes as described in Section \ref{sec:pvar} has the smooth quotient property by Proposition \ref{prop:sq}, we may topologize $\Xf(R)$ for any finite-type algebraic stack $\Xf$ over $R$ and the results of this paper apply to the topologization.
\end{rmk}

\begin{lem}\label{lem:prodsets}
  Let $R=\prod_i R_i$ be a product of fields and complete discrete valuation rings. Let $\Xf$ be a quasi-separated Deligne-Mumford stack over $R$. Then $\Xf(R)=\prod_i \Xf(R_i)$ as topological spaces.
\end{lem}

\begin{proof}
  
  We first assume that $\Xf$ is an algebraic space.
  
  By Proposition \ref{prop:suffconnDM} we may choose a smooth cover $Z\to \Xf$ by a separated scheme such that $Z(R)\to \Xf(R)$ is surjective. Let $T=Z\times_{\Xf} Z$. We may describe $\Xf(R)$ as the coequalizer of the two projections $T(R)\rightrightarrows Z(R)$ by Proposition \ref{quoquo}. Similarly we may describe $\Xf(R_i)$ as the coequalizer of $T(R_i)\rightrightarrows Z(R_i)$. By Remark \ref{rem:BC}  $T(R)=\prod_i T(R_i)$ and $\prod_i Z(R_i)=Z$. Therefore, we may also decribe $\Xf(R)$ may also be described as the coequalizer of $\prod_i T(R_i)\rightrightarrows \prod_i Z(R_i)$, but as coequalizers commute with products this implies that $\Xf(R)=\prod_i \Xf(R_i)$.

  The proof when $\Xf$ is a Deligne-Mumford stack is the same, but $Z\times_{\Xf}Z=T$ is only an algebraic space and we use the algebraic space case to deduce that $T(R)=\prod_i T(R_i)$.
\end{proof}

\begin{prop}
  Let $S$ be the direct limit of sufficiently disconnected rings. Then $S$ is sufficiently disconneted.
\end{prop}
\begin{proof}
  Write $S= \varinjlim_i T_i$, where each $T_i$ is sufficiently disconnected. We first check property $(1)$ of sufficiently disconnceted. Let $M$ be a finitely generated projective module over $S$. As finitely generated projective modules are finitely presented, there must exist an $i$ and finitely presented projective module $M_i$ over $T_i$ such that $M\cong M_i\otimes_{T_i}S$. As $T_i$ is sufficiently disconnected, $M_i$ is free, so $M$ is free.

  Now we check property $(2)$. Let $R'$ be a faithfully flat \'etale $S$-algebra. Then as $R'$ is a finitely presentedly $S$-algebra, there is an $i$ and a faithfully flat \'etale  $T_i$-algebra $R_i'$ such that $R'\cong R_i'\otimes_{T_i}S$. As $T_i$ is sufficiently disconnected, there is a $R_i'$-algebra $R_i''$ that is finite \'etale as a $T_i$-algebra. Then $R''=R''_i\otimes_{T_i}S$ is a $R'$-algebra that is finite \'etale as an $S$-algebra. Thus we conclude that $S$ is sufficently disconnected.
\end{proof}

\section{Stacks over the adeles}\label{sec:adeles}
In this section $k$ will be a global field. For any place, $v$, of $k$, $k_v$ will denote the completion at $v$, and if $v$ is nonarchimedean $\mathscr{O}_v$ will denote the valuation ring of $k_v$. Through the section $I$ will be a set of places of $k$ and $R=\prod_{v\in I}' (k_v,\mathscr{O}_v)$. For any finite set of places $J\subseteq I$, we will let $R_J$ denote $\prod_{v\in J}k_v\times \prod_{v\in I\setminus J}\mathscr{O}_v$.

If $\Yf$ is an algebraic stack over $R_J$, then by Remark \ref{prt} we may topologize $\Yf(R_J)$. If $J'\supseteq J$, we have a natural open inclusion $R_{J}\to R_{J'}$ which by Lemma \ref{ff1} induces a natural map $\Yf(R_J)\to \Yf(R_{J'})$. Now assume futhermore that $\Yf$ is quasi-separated. By Lemma \ref{lem:prodsets}, $\Yf(R_J)\cong \prod_{v\in J}\Yf(k_v)\times \prod_{v\in I\setminus J}\Yf(\mathscr{O}_v)$, and similarly $\Yf(R_{J'})\cong \prod_{v\in J'}\Yf(k_v)\times \prod_{v\in I\setminus J}\Yf(\mathscr{O}_v)$. Then by Theorem \ref{thm:opensubring}, $\Yf(\mathscr{O}_v)$ is open in $\Yf(k_v)$. This implies $\Yf(R_J)$ is open in $\Yf(R_{J'})$.

If $\Xf$ is any finitely presented stack over $R$, there is a finite $J\subseteq I$ and stack $\Yf$ over $R_J$ such that $\Xf\cong \Yf\times_{R_J} \Spec R$. Furthermore, as $\Xf$ is finitely presentated, $\Xf(R)=\colim_{J'\supseteq J}\Yf(R_{J'})$ where the colimit is over finite subsets of $I$ containing $J$ (see \cite{LMB}[Proposition 4.18])

\begin{defi}
  For an algebraic $R$-stack $\Xf$ of finite presentation, there exists a finite subset $J\subseteq I$ and $R_J$-scheme $\Yf$ such that $\Xf\cong \Yf\times_{\Spec R_j}\Spec R$. We define $\Xf(R)=\varinjlim_{J'\supseteq J}\Yf(R_{J'})$ as a topological space.
\end{defi}

\begin{prop}
  Let $\Xf$ be a separated Deligne-Mumford stack of finite presentation over $R$, then \[\Xf(R)=\prod{}'(\Xf(k_v),\Xf(\mathscr{O}_v))\] as topological spaces.
\end{prop}
\begin{proof}
  Let $J\subseteq I$ be a finite subset such that there exists $\Yf$ an algebraic stack over $R_J$ with $\Xf\cong \Yf\times_{R_J}R$.

  Using Proposition \ref{prop:suffconnDM} find $Z\to \Yf$ a smooth cover by a separated scheme such that, $Z(R_{J'})\to \Yf(R_{J'})$ is surjective for all $J'\supseteq J$.

  For such $J'$ by Lemma \ref{lem:prodsets}, $\Yf(R_{J'})=\prod_{v\in J'}\Yf(k_v)\times\prod_{v\in I\setminus   J'}\Yf(\mathscr{O}_v)$ as topological spaces. Note that $\Yf(k_v)=\Xf(k_v)$ and $\Yf(\mathscr{O}_v)=\Xf(\mathscr{O}_v)$. Therefore, $\Yf(R_{J'})=\prod_{v\in J'}\Xf(k_v)\times\prod_{v\in I\setminus   J'}\Xf(\mathscr{O}_v)$. Taking a colimit over $J'$ yields the result.
\end{proof}

\begin{defi}
  Let $k$ be a field of characteristic zero. A stacky curve over $k$ is an algebraic stack $\Xf$ which is smooth, irreducible, 1-dimensional, and Deligne-Mumford and such that there is a dense Zariski open $U\subseteq \Xf$ such that $U$ is a scheme.
\end{defi}

\begin{defi}
  Let $k$ be a field of characteristic $0$, and let $\Xf$ be a stacky curve over $k$. Let $\Xf_{\text{coarse}}$ be the coarse moduli space and let $\Xf\to \Xf_{\text{coarse}}$ be the natural morphism. Let $P_1,\ldots,P_n\in \Xf_{\text{coarse}}(\ov{k})$ over which $\Xf\to \Xf_{\text{coarse}}$ is not an isomorphism. For each $1\leq i\leq n$ let $e_i$ be the order of the stabilizer over $P_i$. Then we define the Euler characteristic of $\Xf$ to be \[\chi(\Xf)= \chi(\Xf_{\text{coarse}})-n+\sum_i \frac{1}{e_i}.\]

    Define the genus of $\Xf$ by \[g(\Xf)=\frac{2-\chi(\Xf)}{2}.\]
\end{defi}

\begin{lem}\label{density}
  Let $\Xf$ be a stacky curve over a essentially analytic field $k$ of characteristic $0$. Let $U\subseteq \Xf$ be a dense open substack which is a scheme. Then $U(k)$ is dense in $\Xf(k)$.
\end{lem}
\begin{proof}
  Set $X=\Xf_{\text{coarse}}$. Note that $U\to X$ is an open inclusion. Let $C$ be the complement of $U$ in $X$ with the reduced scheme structure; this is a finite scheme. By Proposition \ref{prop:suffconnDM}, we may find an \'etale cover $f:Z\to \Xf$ by a scheme, such that $Z(k)\to \Xf(k)$ is surjective. Note that $ Z$ must also be $1$-dimensional and smooth.

  Let $x\in \Xf(k)$, and let $V\subseteq \Xf(k)$ be any open subset containing $x$. The goal is to show that $V\cap U(k)$ is nonempty. Let $z\in Z(k)$ be any preimage of $x$.

  Let $Z'\subseteq Z$ be the connected component of $Z$ containing $z$. We will first show that composite $Z'\to X$ is nonconstant. Assume by sake of contradiction that it is constant. Because $Z'\to X$ is constant, $Z'$ must map to a unique point of $X$, and furthermore since $Z'(k)$ is nonempty, the point in  the image of $Z'$ must have residue field $k$. Let $p\in X(\ov{k})$ be a geometric point localized at the image of $Z'$. As $X$ is a coarse moduli space, $\Xf(\ov{k})\to X(\ov{k})$ is bijective, so we may lift $p\in X(\ov{k})$ to some map $q:\Spec \ov{k}\to \Xf$. Set $\mathscr{T}=Z'\times_{\Xf,q}\Spec \ov{k}$. Consider the map $\mathscr{T}(\ov{k})\to Z'(\ov{k})$. This must be surjective as every  point in $Z'(\ov{k})$ maps to $p\in X(\ov{k})$ and therefore must have image in $\Xf$ isomorphic to $q$. Therefore, $ \mathscr{T}(\ov{k})$ is infinite. On the other hand, $\mathscr{T}$ is \'etale over $\Spec \ov{k}$, so $\mathscr{T}(\ov{k})$ is finite. This is a contradiction, so $Z'\to X$ must be nonconstant.

  As $k$ is essentially analytic, near $z\in Z'(k)$ the space $Z'(k)$ is homeomorphic to an an open subset of $k$. In particular any nonempty open $W\subseteq Z'(k)$ has infinitely many points and thus is Zariski dense in $Z'$. If the open $W$ under the composite $Z'(k)\to \Xf(k)\to X(k)$ lands in $C(k)$, then as $W$ is Zariski dense in $Z'$, $Z'\to X$ lands in $C$. However, as $Z'\to X$ is nonconstant, this cannot be the case. We conclude that the image of $W \to X(k)$ intersects $U(k)$.

  The set $W=f^{-1}(V)\cap Z'(k)$ is an open of $Z'(k)$ containing $z$. By the last paragraph, this open must have a point mapping to $U(k)$, so $V\cap U(k)$ must be nonempty as desired, so $U(k)$ is Zariski dense.
\end{proof}

\begin{thm}
  Let $k$ be a number field. Let $R$ be obtained from the ring of adeles of $k$ by removing the factor corresponding to one place. Let $\Xf$ be a stacky curve over $k$ of genus less than $1/2$. Then $\Xf(k)$ is dense in $\Xf(R)$.
\end{thm}

\begin{proof}
  In the notation of the section, $I$ is all but one place of $k$
  
  We may assume $\Xf(R)$ is nonempty. By \cite{BP}, the genus hypothesis on $\Xf$ implies that $\Xf_{\text{coarse}}\cong \P^1_k$ and the map $\Xf\to \Xf_{\text{coarse}}$ is an isomorphism away from one point, which we assume to be $\infty$. We must show that any nonempty open $V\subseteq \Xf(R')$ contains a point of $\Xf(k)$.

  There is a finite set of places $S$ of $k$ such that $\Xf$ descends to an $\mathscr{O}_{k,S}$-stack $\Xf'$. We can enlarge $S$ so that $\Xf'_{\text{coarse}}\cong \P^1_{\mathscr{O}_{k,S}}$. We view $\A^1_{\mathscr{O}_{k,S}}$ as an open substack in $\Xf'$. For any $S'$ a set of places, let $R^{S'}=\prod'_{v\in S'}(k_v,\mathscr{O}_v)$.

  Now there must be a finite set of places $S'\supseteq S$ and nonempty open subsets $V_{v}\subseteq \Xf'(k_v)$ for $v\in S'$ such that $V\supseteq \prod_{v\in S'}V_{v}\times \Xf(R^{S'})\subseteq \prod_v \Xf(k_v)=\Xf(R)$. By Lemma \ref{density}, we can shrink each $V_{v}$ to be nonempty have and the property that $V_{v}\subseteq \A^1(k_v)$. Then $V$ contains $\prod_{v\in S'}V_{v}\times \A^1(R^{S'})\subseteq \A^1(R) $, and strong approximation for the adeles guarantees that this has a $k$-point. This completes the proof.  
\end{proof}

\section{Appendix}
\begin{lem}\label{top:1}
  Let $I$ be an index set. For each $i\in I$, let $X_i$ be a topological space and $Z_i\subseteq X_i$ a subspace.

  Then, the natural map $\prod_{i\in I}Z_i\to \prod_{i\in I}X_i$ is a homeomorphism onto its image where we give each product the product topology.
\end{lem}

\begin{proof}
  Let $T$ be the image of $\prod_{i\in I}Z_i$ in $\prod_{i\in I}X_i$. The map $\prod_{i\in I}Z_i\to T$ is bijective. Thus the topology on $T$ is coarser than that on $\prod_{i\in I}Z_i$. Therefore, we must show that any open $U\subseteq \prod_{i\in I}Z_i$ has open image in $T$.

  For each $i\in I$, we choose an open $U_i\subseteq Z_i$ such that  $U_i=Z_i$ for all but finitely many $i$. The set $\prod_{i\in I}U_i$ is open in $\prod_{i\in I}Z_i$, and opens of this type form a basis for the topology on $\prod_{i\in I}Z_i$. Therefore, it suffices to show that for any open of $\prod_{i\in I}Z_i$ of this form has open image in $T$.

  Let us choose such an open and thus such $U_i$. As $Z_i$ is a subspace of $X_i$, we may find open $V_i\subseteq X_i$ such that $V_i\cap Z_i= U_i$; if $U_i=Z_i$ we may take $V_i=X_i$. For each $i\in I$, choose such a $V_i$ while choosing $V_i=X_i$ if $U_i=Z_i$. Then $V=\prod_{i\in I}V_i$ is an open subset of $\prod_{i\in I}X_i$. Futhermore, $V\cap T$ is the image of $\prod_{i\in I}U_i$. Thus the image of $\prod_{i\in I}U_i$ in $T$ is open as desired.
\end{proof}

\begin{lem}\label{top:c}
  Let $I$ be an index set. For each $i\in I$, let $X_i\to Y_i$ be a map of topological spaces which is a quotient onto  its image.

  Then the natural map $\prod_{i\in I}X_i\to  \prod_{i\in I}Y_i$ is a quotient onto its image.
\end{lem}

\begin{proof}
  For each $i$, let $R_i=X_i\times_{Y_i}X_i$. Then $Y_i$ is the coequalizer of $R_i\rightrightarrows X_i$ where the maps are the projections. As products commute with connected colimits, we conclude that $\prod_{i\in I}$ is the equalizer of $\prod_{i\in I}R_i\rightrightarrows \prod_{i\in I}X_i$, and thus $\prod_{i\in I} Y_i $ is a quotient of $\prod_{i\in I}X_i$.
\end{proof}

\begin{prop}\label{top:a}
  Let $X,Y,Z,$ and $W$ be topological spaces. Let 
  \begin{center}
    \begin{tikzcd}
      X \ar[r,"f"]\ar[d,"s"]& Y\ar[d,"g"]\\
      Z\ar[r,"t"]& W
    \end{tikzcd}
  \end{center}
  be a commutative diagram of the underlying sets. If $f$, $g$, and $s$ are continuous  and if futhermore, $s$ is a quotient map, then $t$ is continuous.
\end{prop}
\begin{proof}
  Let $U\subseteq W$ be an open subset. Let $U'=f^{-1}(g^{-1}(U))$. As $f$ and $g$ are continuous, $U'\subseteq X$ is open.

  By the commutativity of the diagram, we also have $U'=s^{-1}(t^{-1}(U))$. Therefore, $U'=s^{-1}(s(U'))$. By the definition of the quotient topology, as $U'$ is open and $U'=s^{-1}(s(U'))$, $s(U')\subseteq Z$ is open. But as $s$ is surjective, $s(U')=s(s^{-1}(t^{-1}(U)))=t^{-1}(U)$. Therefore, $t^{-1}(U)$ is open. As $U$ was arbitrary, we conclude that $t$ is continuous as desired
\end{proof}

\begin{prop}\label{top:b}
  Let $X,Y,Z,$ and $W$ be topological spaces. Let 
  \begin{center}
    \begin{tikzcd}
      X \ar[r,"f"]\ar[d,"s"]& Y\ar[d,"g"]\\
      Z\ar[r,"t"]& W
    \end{tikzcd}
  \end{center}
  be commutative diagram of topological spaces. If $f$ and $g$ are quotients, then so is $t$.
\end{prop}

\begin{proof}

  As $f$ and $g$ are quotients, so is $g\circ f$. This means that $W$ has the finest topology such that $X\to W$ is continuous.

  If $t$ is not a quotient, we may put a finer topology on $W$ so that $Z\to W$ is continuous. If this were the case, then that finer topology on $W$ would make the composite $X\to Z\to W$ be continuous. However, $W$ already has the finest topology making that composite continuous. We conclude $t$ is a quotient.

\end{proof}

\begin{prop}
  Let $U$ and $V$ be topological spaces and $f:U\to V$ a continuous map.

  Let $I$ be an index set and let $\{U_i\}_{i\in I}$ be an open cover of $U$. If for every $i$, $U_i\to V$ is a quotient onto its image, then $U\to V$ is a quotient onto its image.
\end{prop}
\begin{proof}
  Let $Z$ be the image of $U$ in $V$, and let $Z_i$ be the image of $U_i$ for each $i\in I$. To show that $U\to V$ is a quotient onto $Z$, we must show that if $W\subseteq U$ is an open such that $f^{-1}(f(W))=W$, then $f(W)$ is open in $Z$.
\end{proof}

\begin{lem}\label{lem:22223}
  Let \[
    \begin{tikzcd}
      X \ar[rr,"f"]\ar[dr,"g"] & & Y\ar[dl,"h"]\\
      & Z & 
    \end{tikzcd}\]
  be a commutative diagram of schemes. If $g$ is finite \'etale and $h$ is separated and \'etale, then $f$ is \'etale and the scheme theoretic image of $f$ is finite \'etale  over $Z$.
  
  Therefore, if $R$ is a ring, $R'$ an \'etale $R$-algebra, and $S$ a $R'$-algebra which is finite \'etale as an $R$-algebra, the image $T$ of $R'$ in $S$ is finite \'etale over $R$.
\end{lem}
\begin{proof}
  Consider $X\times Y$ let the projection to $X$ be  $\pi_1$ and the projection to $Y$ be $\pi_2$. Note that $\pi_1$ and  $\pi_2$ are pullbacks of $h$ and $g$  respectively, so are \'etale.

  The graph $\Gamma_f$ provides a section to $\pi_1$, and thus as a section of an \'etale  map $X\xrightarrow{\Gamma_f} X\times_Z Y$ has open image. The composite  $X\xrightarrow{\Gamma_f}X\times_Z Y \xrightarrow{\pi_2} Y$ is $f$, and the first map is an open inclusion and the second \'etale, so we conclude that $f$ is \'etale.

  Now $f$ has open image because it is \'etale, and it has closed image because $X$ is finite. Therefore, $f$ is surjective onto a clopen subset of $Y$. This clopen subset must be the scheme theoretic image of  $Y$ and since $Y\to Z$ is \'etale and the scheme theoretic image of $X$  is open, we conclude that the restriction of $h$ to the scheme theoretic image of $X$ is \'etale.
\end{proof}

\begin{rmk}\label{rem:22223}
  Lemma \ref{lem:22223} holds when $Y$ is only an algebraic space.
\end{rmk}

\section*{Acknowledgments}
I would to thank Bjorn Poonen for suggesting this problem and for his guidance.

\bibliography{main}
\bibliographystyle{alpha}
\end{document}